\documentclass[10pt, reqno]{amsart}

\usepackage{amssymb}
\usepackage{amsfonts}
\usepackage{amsthm}
\usepackage{amsmath}
\usepackage{bbm}
\usepackage{xcolor}
\usepackage{mathtools}
\usepackage{bigints}
\usepackage{enumerate}
\usepackage{tikz}

\usepackage{thmtools}
\usepackage{thm-restate}

\numberwithin{equation}{section}
\allowdisplaybreaks

\newtheorem{theorem}{Theorem}[section]

\newtheorem{prop}[theorem]{Proposition}
\newtheorem{lemma}[theorem]{Lemma}

\newtheorem{conjecture}[theorem]{Conjecture}
\theoremstyle{definition}

\theoremstyle{remark}
\newtheorem{remark}[theorem]{Remark}

\newcommand{\R}{\mathbb{R}}
\newcommand{\T}{\mathbb{T}}

\renewcommand{\hat}{\widehat}

\begin{document}
\title[Polynomial partitioning and restriction study guide]{A study guide for ``A restriction estimate using polynomial partitioning''}

\author{John Green}
\address{University of Pennsylvania}
\email{jdgreen@sas.upenn.edu}

\author{Terry Harris}
\address{University of Wisconsin--Madison}
\email{terry.harris@wisc.edu}

\author{Kaiyi Huang}
\address{University of Wisconsin--Madison}
\email{khuang247@wisc.edu}

\author{Arian Nadjimzadah}
\address{University of California--Los Angeles}
\email{anad@math.ucla.edu}

\begin{abstract}
    This manuscript is intended as an accompaniment to Guth's ``A restriction estimate using polynomial partitioning". We begin by summarizing the core ideas of the proof, elaborating the history and development of the techniques therein. From there, we provide supplementary details on some of the standard methods and more technical arguments which may be unfamiliar or less accessible to readers not yet acquainted with the paper. We also provide a summary of some more recent developments since the publication of Guth's work.
\end{abstract}

\maketitle

\section{Introduction}
The goal of this manuscript is to discuss Guth's paper ``A restriction estimate using polynomial partitioning" \cite{restest}. In this opening section, we will discuss the history and motivations leading up to the paper, and give a high-level overview of the argument. In the following sections, we will discuss the limitations of the argument, and how later work has refined and built upon it. We will also provide some supplementary discussion on select topics to assist first-time readers of the paper.

The main result is the following:

\begin{theorem}\label{thm: restriction_result}
If $S\subseteq\mathbb{R}^3$ is a compact $C^\infty$ hypersurface (possibly with boundary) having strictly positive second fundamental form, then for all $p>3.25$ and $f\in L^\infty(S)$, we have
$$\|E_Sf\|_{L^p(\mathbb{R}^3)}\leq C(p,S)\|f\|_{L^\infty(S)},$$
where $E_S$ is the extension operator,
$$E_Sf(x):=\int_Se^{i\omega x}f(\omega)\,d\text{vol}_S(\omega).$$ 
\end{theorem}

Note that the conjectured range is $p>3$, but this result presented an improvement on the known range. The key new ingredient that enabled this result was the polynomial partitioning method introduced by Guth and Katz \cite{distdist}.

We begin this discussion by tracing some key moments in the genealogy of this method so as to properly motivate its development and application in harmonic analysis.

The first of these we shall discuss is the introduction of partitioning methods in incidence geometry by Clarkson, Edelsbrunner, Guibas, Sharir, and Welzl \cite{combcomp}. Here is a very heuristic introduction to the philosophy of this method.

In a typical incidence geometry problem, we might want to count the number of intersection points of some collection of lines, satisfying certain properties. It is natural to try to break this into subproblems which can be solved separately, perhaps by considering subcollections of the lines, for instance. If we can find some procedure for doing so which preserves exploitable geometric structure in the problem, then it suffices to solve these subproblems (which we may solve e.g. by induction).

The high-level idea of partitioning is to find procedures for dividing into such subproblems which, in the aforementioned paper and several which followed, was done using probabilistic arguments. Polynomial partitioning instead uses elementary topological arguments to select a ``good" partition.

Before we detail this, though, it is pertinent to consider the history between the introduction of partitioning and the introduction of polynomial partitioning. Firstly (as will be relevant to later discussion), Bourgain \cite{btypemax} established combinatorial estimates about overlapping tubes pointing in different directions, arising in connection to the Kakeya maximal function. He then applied these estimates also to obtain restriction estimates by making use of an associated wave packet decomposition.

Wolff \cite{kakeyacirc} noted the similarities between point-line incidence problems in incidence geometry and the problems on overlapping tubes Bourgain had considered, and leveraged methods from incidence geometry to obtain estimates in harmonic analysis. In exploring this connection, Wolff \cite{reckakeya} also introduced the finite field analogue of the Kakeya problem.

\subsection{Motivation: Finite Field Kakeya and Polynomial methods}

Let $\mathbb{F}_q$ denote a finite field of $q$ elements. By a Kakeya set we mean a set $K\subseteq\mathbb{F}_q^n$ which contains a line in every direction. That is, for every direction $x\in\mathbb{F}_q^n$, there exists $y\in\mathbb{F}_q^n$ such that
$$L_{x,y}:=\{a\cdot x+y:a\in\mathbb{F}\}$$
is a subset of $K$. The finite field Kakeya conjecture is that there should be a positive constant $C_n$, independent of $q$, such that every such Kakeya set has size at least $C_nq^n$.

This was later solved by Dvir \cite{ffkakeya}. The proof goes as follows:

Let $d=q-1$. Suppose for a contradiction that
$$|K|<\left(\begin{matrix}
d+n\\
n
\end{matrix}\right)\sim_n q^n.$$
The number in the middle is simply the dimension of the space $\mathcal{P}_d$ of polynomials of degree at most $d$. The map sending polynomials in $\mathcal{P}_d$ to the vector of its values at the points of $K$ has rank at most $|K|$, so by assumption its kernel is non-zero, i.e. there exists a non-zero polynomial $P\in\mathcal{P}_d$ which is $0$ on $K$. Denote its degree by $d'\leq d$.

Let $Q$ be the homogeneous part of $P$, that is, the sum of the degree $d'$ terms of $P$. Let $x\in\mathbb{F}_q^n$, and $y\in\mathbb{F}_q^n$ be such that $L_{x,y}\subseteq K$. We have that $P_{x,y}(t):=P(t\cdot x+y)$ has degree at most $d'<q$, but has $q$ zeroes, hence it is the zero polynomial. In particular, the coefficient of $t^{d'}$ in $P_{x,y}(t)$ is $0$, but by expanding $P(t\cdot x+y)$ we easily see that this coefficient is simply $Q(x)$. Hence $Q$ is $0$ at every $x$, and it follows that we must have $P=0$. The bound follows.

The proof essentially works because the dimension of the space of degree $d$ polynomials is comparable to $d^n$, but along lines this drops to $d$. This numerology proves to be incredibly efficient for studying incidence phenomena with lines.

The next step of the historical developments was to try to find ways to leverage this highly efficient numerology to study continuum problems, which leads us to Guth's \cite{multkakeya} continuum version of the ``polynomial method". The basic observation was the ``polynomial ham sandwich theorem", which states that ``given finite volume open sets $U_1,\ldots,U_N$, where $N$ is $\dim(\mathcal{P}_d)-1$, there exists a degree $d$ polynomial whose zero set bisects each $U_i$".

We give a slightly different result later. Both make use of the Borsuk-Ulam Theorem, and the other methods in Guth's paper on multilinear Kakeya \cite{multkakeya} use further machinery from algebraic topology, though we note that both the Borsuk-Ulam Theorem and the remainder of the proof can be understood without the use of algebraic topology, see \cite{borsukulam} and \cite{multkakeya2}.

\subsection{Polynomial partitioning}
The core principle of polynomial partitioning is to combine the partitioning methods of incidence geometry with the high-efficiency numerology of the polynomial methods. The technique was developed by Guth and Katz \cite{distdist} and was implemented in their study of the distinct distances problem. Instead of the polynomial ham sandwich theorem given above, they proved a similar theorem for finite collections of points, also using the Borsuk-Ulam Theorem:

\begin{theorem}
Let $X\subseteq\mathbb{R}^n$ be finite. Then for each $D\geq 1$, there is a polynomial $P$ of degree at most $D$ so that each component of $\mathbb{R}^n\setminus Z(P)$ contains at most $\lesssim_n D^{-n}|X|$ points of $X$.
\end{theorem}

Here $Z(P)$ denotes the zero set of $P$. Combined with the classical observation of Milnor \cite{bettino} that every degree $D$ polynomial divides $\mathbb{R}^n$ into at most $\sim_n D^n$ components, this means that every cell contains at most an even share of the points of $X$, but it could also be that the majority of points of $X$ in fact lie on the variety $Z(P)$.

The technique is best illustrated with the following simple example, a weakened form of a result established by Guth and Katz \cite{distdist}:

\begin{theorem}
For any $\varepsilon>0$, there is a degree $D$ so that the following is true. Suppose $\mathcal{L}$ is a set of $L$ distinct lines in $\mathbb{R}^3$ such that at most $S$ lines lie in any degree $D$ algebraic surface. Then the set of intersection points of lines in $\mathcal{L}$, $I$, satisfies
$$|I|\leq C(\varepsilon,S)L^{3/2+\varepsilon}.$$
\end{theorem}

\textit{Remark.} The assumption that lines do not cluster in algebraic surfaces is necessary. For instance, if we allow all the lines to lie in a plane, the best possible bound is $\sim L^2$, since then we could take $\sim L/2$ lines parallel to the $x$ axis and $\sim L/2$ lines parallel to the $y$ axis, so every line would intersect every other in a distinct point.

\begin{proof}
Fix $D$ to be chosen later. Apply the partitioning theorem to select $P$ of degree $D$ dividing $\mathbb{R}^3$ into cells $O_i$ containing $\lesssim D^{-3}|I|$ points each.

If the majority of points in $I$ lie in $Z(P)$, then it suffices to estimate the size of $I\cap Z(P)$. Let $\mathcal{L}_Z$ be the set of those lines in $\mathcal{L}$ contained in $Z(P)$.
By assumption, there are at most $S$ lines in $\mathcal{L}_Z$, so the contribution to $I\cap Z(P)$ from pairs of such lines is at most $|\mathcal{L}_Z\times\mathcal{L}_Z|\leq S^2$.
Also, every line in $\mathcal{L}\setminus\mathcal{L}_Z$ intersects $Z(P)$ in at most $D$ points, so these lines contribute at most $DL$ points. Thus $|I\cap Z(P)|\leq DL+S^2$. This is bounded by $C(\varepsilon,S)L^{3/2+\varepsilon}$ if we choose $C(\varepsilon,S)$ large enough.

The alternative is that the cells contain the majority of $I$. In this case we use induction on $L$. We must then have that $\sim D^3$ cells contain $\sim D^{-3}|I|$ points of $I$. Since every line intersects $Z(P)$ at most $D$ times, each enters at most $D+1$ of the cells.

Counting pairs of cells and lines which have the property that the line intersects the cell, we see that there must be one that intersects at most $\lesssim D^{-2}L$ of the lines. Fix one such cell $O$.

Considering only the lines intersecting this $O$ -- call this collection $\mathcal{J}$ -- and the corresponding intersection set $J$ of lines in $\mathcal{J}$, we have by induction hypothesis
$$|I|\lesssim D^3|J|\leq D^3C(\varepsilon,S)|\mathcal{J}|^{3/2+\varepsilon}\lesssim D^3C(\varepsilon,S)(D^{-2}L)^{3/2+\varepsilon}.$$
The induction closes once we choose $D$ so that the implicit constant $C$ satisfies $CD^{-2\varepsilon}<1$.
\end{proof}

Observe that we estimated $|I|$ by partitioning $I$ and studying three types of contribution:
\begin{itemize}
\item Contributions from lines in $Z(P)$
\item Contributions from lines crossing $Z(P)$
\item Contributions from each of the $D^3$ cells, which we handle by induction.
\end{itemize}

This scheme is morally the same as how polynomial partitioning will be used in the proof of the main theorem. Naturally, we will need to replace cardinalities with measures, lines with tubes, and instead of lines lying in or crossing algebraic varieties, we will consider tubes intersecting approximately tangentially and transversally in a neighborhood of an algebraic variety.

\subsection{Overview of the argument: Initial reductions}
In this section, we give a high-level overview of the argument in \cite{restest}. The details are fully elaborated in much of Guth's paper (or otherwise routine) so it is not worth reviewing as such here. However, it may be helpful to give a zoomed-out view of the argument, and motivate each step.

Firstly, by splitting $S$ into pieces and rotating and performing a parabolic rescaling on each, Guth is able to reduce to the case where $S$ is a sufficiently small pertubation (up to order $C^L$) of a compact piece of the graph of a paraboloid. Then, by Tao's epsilon-removal theorem (see \cite{brrest}, also \cite{remarks} for relevant comments), the desired estimates reduce to establishing
$$\|E_Sf\|_{L^p(B_R)}\leq C(p,\varepsilon)R^\varepsilon\|f\|_{L^\infty(S)}$$
for some small $\varepsilon>0$. From now on, we shall write $Ef$ for $E_Sf$.

The next reduction is the reduction to an estimate for broad points. Suppose that we have divided our surface into caps $\tau$, and let $\alpha\in(0,1)$. Let $f_\tau:=f\chi_\tau$. We call a point $x$ $\alpha$-broad for $Ef$ if
$$\max_\tau|Ef_\tau(x)|\leq\alpha|Ef(x)|.$$
For notational convenience, we define the ``broad part" $\text{Br}_\alpha Ef(x)$ to be $|Ef(x)|$ when $x$ is $\alpha$-broad and $0$ otherwise. If $x$ is not $\alpha$-broad, then there is $\tau$ such that $|Ef(x)|\leq \alpha |Ef_\tau(x)|$.

Supposing we have a sufficiently ``good" estimate on the broad part, the argument then proceeds by a routine induction on $R$. We will estimate separately the contribution from the broad part and the $Ef_\tau$, and to handle the latter we will use a parabolic rescaling/change of variables to reduce to the estimate on a smaller ball. The number of caps (as well as their size), and the number $\alpha$, will appear in the resulting expression and in the argument these parameters are such that the induction closes.

The required broad part estimate is the following theorem:

\begin{theorem}\label{broadest}
For any $\varepsilon>0$, there exists $K=K(\varepsilon)$ and $L=L(\varepsilon)$ with $K(\varepsilon)\rightarrow\infty$ as $\varepsilon\rightarrow 0$ such that the following holds. Suppose that $S$ is a $C^L$ perturbation of a paraboloid (as described above) and is partitioned into caps $\tau$ of diameter $\sim K^{-1}$. Then, for any radius $R\geq 1$,
$$\|\text{Br}_{K^{-\varepsilon}} Ef\|_{L^{3.25}(B_R)}\leq C_\varepsilon R^\varepsilon \|f\|_2^{12/13}\|f\|_\infty^{1/13}.$$
\end{theorem}

We will in fact prove something slightly more general, which is set up in an appropriate inductive framework inspired by work of Solymosi and Tao \cite{highdiminc}. We will assume instead that the caps $\tau$ are graphs over balls $B(\omega_\tau,r)$ which are now allowed to overlap, but have $K^{-1}$-separated centres, and that $f$ decomposes as a sum of functions $f_\tau$ which are supported in $\tau$. We define the multiplicity $\mu$ of the cover by saying that $r\in[K^{-1},\mu^{1/2}K^{-1}]$, from which it follows that each $x$ lies in at most $O(\mu)$ caps $\tau$.

\begin{theorem}\label{induction}
For each $\varepsilon>0$, there exist $K=K(\varepsilon)$, $L=L(\varepsilon)$ and $\delta_{trans}=\delta_{trans}(\varepsilon)\in(0,\varepsilon)$ with $K(\varepsilon)\rightarrow\infty$ as $\varepsilon\rightarrow 0$ so that for every $S$ which is a $C^L$ perturbation of a paraboloid (in a particular sense), and the caps $\tau$ cover $S$ with multiplicity at most $\mu$, and $f_\tau$ are as above, the following holds.

Suppose $\alpha\geq K^{-\varepsilon}$, and that for every $\omega\in S$ and any cap $\tau$, we have the averaged integrals
$$\oint_{B(\omega,R^{-1/2})\cap S}|f_\tau|^2\leq 1.$$
Then
$$\int_{B_R}(\text{Br}_\alpha Ef)^{3.25}\leq C_\varepsilon R^\varepsilon\left(\sum_\tau \int_S|f_\tau|^2\right)^{(3/2)+\varepsilon}R^{\delta_{trans}\log(K^\varepsilon\alpha\mu)}.$$
\end{theorem}

The previous broad estimate follows trivially. Note that we use local $L^2$ averages instead of an $L^\infty$ norm for the induction to work.

Now, to establish this estimate, we introduce a wave packet decomposition in a fairly routine way. This wave packet decomposition will be associated to a family of tubes, and these tubes will be constructed to be adapted to the surface $S$ in a natural way, and the curvature of $S$ will then force these tubes to point in different directions. Leveraging this, along with polynomial partitioning to control their incidences/overlaps, we will deduce the desired bounds.

\subsection{Overview: The wave packet decomposition}\label{wavepacket}
We decompose $S$ into caps $\theta$ which are the pieces of the graph of a function $h$ over balls $B(\omega_\theta,r)$ as above, except with radii $R^{-1/2}$. By $m\theta$ we will denote the enlarged cap given as the graph over $B(\omega_\theta,mr)$. For a cap $\theta$, denote by $v_\theta$ the unit normal to $S$ at $h(\omega_\theta)$.

Let $\delta>0$ and let $\mathbb{T}(\theta)$ be a finitely overlapping collection of cylindrical tubes with central axis parallel to $v_\theta$, length $\sim R$ and radius $R^{1/2+\delta}$, covering $B_R$. Write $\mathbb{T}=\cup_\theta\mathbb{T}(\theta)$.

By $\theta'$ we will denote a subset of $\theta$ such that the union of the $\theta'$ is $S$ and the $\theta'$ are disjoint. Denote by $f_{\theta}$ the restriction to $\theta'$. 

We can view $S\cap 5\theta$ as a graph by choosing coordinates such that $v_\theta$ is parallel to the third coordinate direction. We can thus think of functions on $S\cap 5\theta$ as functions on $\mathbb{R}^2$ by using projection onto the first coordinates as a change of variable, and vice versa. We shall denote by $F_\theta$ the function on $\mathbb{R}^2$ corresponding to $f_\theta$ on $S\cap 5\theta$.

Let $\tilde{\mathbb{T}}(\theta)$ denote a boundedly overlapping collection of tubes parallel to $v_\theta$ extending $\mathbb{T}(\theta)$ and covering all of $\mathbb{R}^2$. Projecting the tubes $T\in\tilde{\mathbb{T}}(\theta)$ onto the first two coordinates gives balls $B$ covering the ball of radius $R$ in $\mathbb{R}^2$.

We take a partition of unity subordinate to these balls, and denote the piece corresponding to $T$ by $\phi_T$. Doing this in a standard way, we may also assume by the uncertainty principle that
$$|\hat{\phi}_T(\omega)|\lesssim \text{Area}(B)(1+R^{1/2+\delta}|\omega|)^{10^6}.$$
Also, we have
$$F_\theta=\sum_{T\in\tilde{\mathbb{T}}(\theta)}\hat{\phi}_T*F_\theta.$$
We also take a smooth cutoff $\psi_\theta$ equal to $1$ on $2\theta$ and supported in $3\theta$, and we use $\Psi_\theta$ to denote its projection to $\mathbb{R}^2$ using the change of variable associated to $v_\theta$.

Because $F_\theta$ is supported in a ball of radius $R^{-1/2}$, and $|\hat{\phi}_T|$ decays rapidly outside of $|\omega|\leq R^{-1/2-\delta}$, the contribution to $\hat{\phi}_T*F_\theta$ from outside of the projection of $2\theta$ is negligible, so in practice we can replace $\hat{\phi}_T*F_\theta$ with $\Psi_\theta(\omega)(\hat{\phi}_T*F_\theta)(\omega)$ (we can routinely bound the contributions from the remaining part). We change variables back to transfer $\Psi_\theta(\omega)(\hat{\phi}_T*F_\theta)(\omega)$ to a function on $S\cap 5\theta$, which we call $f_T$.

When analyzing $Ef_\theta(x)=\sum_{T\in\tilde{\mathbb{T}}(\theta)}Ef_T(x)$ for $x$ in the ball $B_R$, we may also routinely bound the contribution from the tubes $T\in\tilde{\mathbb{T}}(\theta)\setminus\mathbb{T}(\theta)$, as $Ef_T$ decays rapidly as $x$ moves away from $T$.

We shall work with the $f_T$ for $T\in\mathbb{T}=\cup_\theta \mathbb{T}_\theta$. We summarize the essential properties of the $f_T$, which are easily established.

\begin{itemize}
\item If $T\in\mathbb{T}_\theta$, then $f_T$ is supported in $3\theta$.
\item If $x\in B_R\setminus T$, then $|Ef_T(x)|\leq R^{-1000}\|f\|_2$.
\item For any $x\in B_R$, $|Ef(x)-\sum_{T\in\mathbb{T}}Ef_T(x)|\leq R^{-1000}\|f\|_2$.
\item (Essential orthogonality) If $T_1,T_2\in\mathbb{T}(\theta)$ are disjoint, then
$\int f_{T_1}\bar{f}_{T_2}\leq R^{-1000}\int_\theta|f|^2$. 
\item $\sum_{T\in\mathbb{T}(\theta)}\int_S|f_T|^2\lesssim\int_\theta|f|^2$.
\end{itemize}

We will use this to (approximately) decompose each $f_\tau$ into $f_{\tau,T}$. These parts will be grouped according to properties of the associated tube and studied independently (in practice this grouping will not be disjoint, but the above properties mean that this essentially contributes a factor of the maximal number of groups each tube lies in, which we shall control using incidence arguments when we apply polynomial partitioning).

\subsection{Setting up polynomial partitioning}
We first establish the precise polynomial partitioning theorem we shall use. To begin, recall the Borsuk-Ulam Theorem \cite{borsukulam}:

\begin{theorem}
Let $F:S^N\rightarrow\mathbb{R}^N$ be continuous and odd. Then $\exists v\in S^N$ with $F(v)=0$.
\end{theorem}

Given $L^1$ functions $W_1,\ldots,W_N$ and a vector space $V$ of polynomials having dimension $N+1$, the function $F$ with $j^{\text{th}}$ component $F_j$ defined by
$$F_j(v)=\int_{v>0}W_j-\int_{v<0}W_j$$
is easily checked to be continuous on $V\setminus\{0\}$, so we may apply Borsuk-Ulam to find a polynomial $v\in V$ which equally with the property that
$$\int_{v>0}W_j=\int_{v<0}W_j$$
for every $j$. Since the dimension of the space of polynomials of degree $D$ is $\sim_n D^n$, it follows that we can fine a polynomial of degree $\sim_n N^{1/n}$ with this property.

Applying this result inductively to the same function, we obtain the following:

\begin{theorem}
Let $W\geq 0$ be a non-zero $L^1$ function on $\mathbb{R}^n$. Then for each $D\geq 1$ there is a polynomial of degree $D$ such that $\mathbb{R}^n\setminus Z(P)$ is a union of $\sim_n D^n$ open sets $O_i$ with $\int_{O_i}W=\int_{O_j}W$ for each $i,j$.
\end{theorem}

In practice, we may use density to assume that $P$ is non-singular, in the sense that $\nabla P\neq 0$ on $Z(P)$, provided we instead assume that $\int_{v>0}W_j$ and $\int_{v<0}W_j$ are comparable (with an implicit constants can be made arbitrarily close to $1$). We obtain the following variant of the theorem:

\begin{theorem}
Let $W\geq 0$ be a non-zero $L^1$ function on $\mathbb{R}^n$. Then for each $D\geq 1$ there is a polynomial of degree $D$ which is a product of non-singular polynomials such that $\mathbb{R}^n\setminus Z(P)$ is a union of $\sim_n D^n$ open sets $O_i$ with $\int_{O_i}W\sim\int_{O_j}W$ for each $i,j$.
\end{theorem}

The implicit constant in the comparison $\int_{O_i}W\sim\int_{O_j}W$ can be chosen as close to $1$ as we like. Consequently, we may assume for each $i$ that
$$\int W\sim D^n\int_{O_i}W,$$
where the implicit constant is independent of the polynomial $P$ given by the preceding theorem.

\subsection{The inductive step}
We now apply the polynomial partitioning theorem with $W=\chi_{B_R}\cdot (\text{Br}_\alpha Ef)^{3.25}$. 
To summarize, we have

\begin{prop}
There exists a degree $D$ polynomial $P$, a product of non-singular polynomials, such that $\mathbb{R}^3\setminus Z(P)$ is a union of $\sim D^3$ open sets $O_i$ with
$$\int_{B_R}(\text{Br}_\alpha Ef)^{3.25}\sim D^3\int_{O_i\cap B_R}(\text{Br}_\alpha Ef)^{3.25}.$$
\end{prop}

We will take $D=R^{\delta_{deg}}$, where $\delta_{deg}\sim\varepsilon^4$. We prove the theorem with $\delta_{trans}=\varepsilon^6$, $\delta=\varepsilon^2$, $K=e^{\varepsilon^{-10}}$. The key facts about these choices are that
\begin{itemize}
\item $\delta_{trans}\ll\delta_{deg}\ll\delta\ll\varepsilon$, and
\item $\delta_{trans}\ll K$, so that $R^{\delta_{trans}\log(10^{-6}K^\varepsilon)}\geq R^{1000}$.
\end{itemize}

We define $W$ to be the $R^{1/2+\delta}$ neighborhood of $Z(P)$,
$$W:=N_{R^{1/2+\delta}}Z(P),$$
and define $O_i':=(O_i\cap B_R)\setminus W$ for each $i$.

Ultimately, once we have appropriate estimates for the contributions to the integral $\int_{B_R}(\text{Br}_\alpha Ef)^{3.25}$ coming fromn $W$ and from each of the reduced cells $O_i'$, the proof of the main theorem will proceed by an induction on $R$ and $\sum_\tau\int|f_\tau|^2$ in a routine way.

We begin by addressing the estimates in the cellular case, that is, when the integral over $\cup_i O_i'$ dominates the integral over $W$.
We associate to each reduced cell $O_i'$ the family of tubes that intersect it,
$$\mathbb{T}_i:=\{T\in\mathbb{T}:T\cap O_i'\neq\emptyset\},$$
and using the wave packet decomposition, we define
$$f_{\tau,i}:=\sum_{T\in\mathbb{T}_i}f_{\tau,T},\quad f_i:=\sum_i f_{\tau,i}.$$
Now, for $x\in O_i'$, it follows straightforwardly from the properties of the wave packet decomposition that for sufficiently large $R$,
$$\text{Br}_\alpha Ef(x)\leq\text{Br}_{2\alpha} Ef_i(x)+R^{-900}\sum_\tau\|f_\tau\|_2.$$
So, up to an error term which is easily dispensed with, we can control the integral of $(\text{Br}_\alpha Ef)^{3.25}$ on $O_i'$ by $(\text{Br}_{2\alpha} Ef_i)^{3.25}$ on $O_i'$. In summary,
$$\int_{B_R}(\text{Br}_\alpha Ef)^{3.25}\lesssim D^3\int_{B_R}(\text{Br}_{2\alpha} Ef_i)^{3.25}.$$
Because the central axis of each tube intersects at most $D+1$ cells $O_i$, it follows that each tube intersects at most $D+1$ of the cells $O_i'$, and so from the properties of the wave packets we have
$$\sum_i\int|f_{\tau,i}|^2\lesssim D\int |f_\tau|^2.$$
Summing in $\tau$, we see that we can find an $i$ with
$$\sum_\tau\int |f_{\tau,i}|^2\lesssim D^{-2}\sum_\tau\int|f_\tau|^2.$$
We are then in a position to apply the inductive hypothesis, and having chosen our parameters appropriately above, the induction closes in this case.

The case where the integral over $W$ dominates is more substantial. We first cover $B_R$ by Balls $B_j$ of radius $R^{1-\delta}$. We define two collections of tubes associated to each ball $B_j$:

\begin{itemize}
\item $\mathbb{T}_{j,tang}$ is the set of tubes $T\in\mathbb{T}$ intersecting $W\cap B_j$ such that if $z$ is any non-singular point of $Z(P)$ lying in $2B_j\cap 10T$, then the angle between the tangent space to $Z(P)$ at $z$ and the central axis of $T$ is at most $R^{-(1/2)+2\delta}$ (the ``tangential tubes").
\item $\mathbb{T}_{j,trans}$ is the set of tubes $T\in\mathbb{T}$ intersecting $W\cap B_j$ such there exists a non-singular point $z$ of $Z(P)$ lying in $2B_j\cap 10T$ with the angle between the tangent space to $Z(P)$ at $z$ and the central axis of $T$ greater than $R^{-(1/2)+2\delta}$ (the ``transversal tubes").
\end{itemize}

It is straightforward to check that a tube $T$ that intersects $W\cap B_j$ will be contained in exactly one of $\mathbb{T}_{j,tang}$ or $\mathbb{T}_{j,trans}$.

We would like to be able to define $f_{j,tang}=\sum_\tau\sum_{T\in\mathbb{T}_{j,tang}}f_{\tau,T}$ and likewise for $f_{j,trans}$, and reduce the analysis on each ball $B_j$ to analyzing $f_{j,tang}$ and $f_{j,trans}$. However, we are estimating the broad part of $Ef$, which does not behave well with respect to this decomposition.

To account for this, we consider for each particular $\alpha$-broad point $x\in B_j\cap W$ a collection $I$ of caps $\tau$ on which the $|Ef_{\tau,j,tang}(x)|$ are at most $K^{-100}|Ef(x)|$.

By considering separately the cases where $I^c$ contains two non-adjacent caps (by non-adjacent, we shall mean $K^{-1}$-separated) and where it does not, we are separately able to bound $|Ef(x)|$ (up to an error term which is easy to control) by either $K^{100}\text{Bil}(Ef_{j,tang})(x)$ (which we shall momentarily define) or $\text{Br}_{2\alpha}Ef_{I,j,trans}(x)$, where $f_{I,j,trans}$ is simply the sum of $f_{\tau,j,trans}$ for $\tau\in I$. Summing over all possible $I$ (of which there are $\sim 2^{K^2}$) in order to apply this for each $x$ introduces a constant which is admissible due to our choice of $K$.

Thus it suffices to estimate the terms $\text{Br}_{2\alpha}Ef_{I,j,trans}$ and the bilinear term $\text{Bil}(Ef_{j,tang})(x)$, where the latter is defined by
$$\text{Bil}(Ef_{j,tang})(x):=\sum\limits_{\tau_1,\tau_2\text{ non-adjacent}}|Ef_{\tau_1,j,tang}(x)|^{1/2}|Ef_{\tau_2,j,tang}(x)|^{1/2}.$$

We first discuss the estimate for the bilinear tangential term. Morally speaking, the problem here is two-dimensional, since the central axes of the tangential tubes intersecting $B_j$ are nearly coplanar. We apply a variant of C\'ordoba's $L^4$ argument \cite{l4arg}, giving an estimate on the $L^4$ norm of $\text{Bil}(Ef_{j,tang})$ on $B_j\cap W$. This can be interpolated with a standard $L^2(B_R)$ bound for the extension operator.

Ultimately, we obtain for $2\leq p\leq 4$ the following bound:
$$\|\text{Bil}(Ef_{j,tang})\|^p_{L^p(B_j\cap W)}\lesssim R^{O(\delta)}R^{\frac{5}{2}-\frac{p}{4}}\left(\sum_\tau \|f_{\tau,j,tang}\|_2^2\right)^{p/2}.$$

For $p\geq 3$, we can ultimately pass to the estimate 
$$\|\text{Bil}(Ef_{j,tang})\|^p_{L^p(B_j\cap W)}\lesssim R^{O(\delta)}R^{\frac{13}{4}-p}\left(\sum_\tau\|f_\tau\|_2^2\right)^{3/2},$$
from which the desired estimate follows when we take $p=3.25$. However, to pass to this estimate, we need one key observation:

\begin{lemma} \label{lem: tangent}
For each $j$, the number of caps $\theta$ for which $\mathbb{T}_{j,tang}\cap\mathbb{T}(\theta)\neq\emptyset$ is $\lesssim R^{(1/2)+O(\delta)}$.
\end{lemma}

This requires a geometric argument adapted from the method of Wongkew \cite{volneigh} for estimating the volumes of neighborhoods of real algebraic varieties, and will be discussed in Section \ref{sec:geometry}.

This leaves the transversal part. Recall that we are in the case where the integral over $W$ dominates, so it suffices to control
$$\int_{W}(\text{Br}_\alpha Ef)^{3.25}.$$
It suffices to estimate the integral over each $B_j\cap W$ and sum. Since there are $\sim R^{3\delta}$ balls, and $\delta\ll\varepsilon$, summing the resulting estimates is okay. So we may focus on estimating each
$$\int_{B_j\cap W}(\text{Br}_\alpha Ef)^{3.25}.$$
This is bounded above (up to an error) by the integrals of $K^{100}\text{Bil}(Ef_{j,tang})^{3.25}$, which we have just dispensed with, and $\sum_I(\text{Br}_\alpha Ef_{I,j,trans})^{3.25}$.

In summary, we need only bound
$$\sum_{j,I}\int_{B_j}(\text{Br}_\alpha Ef_{I,j,trans})^{3.25}.$$
We can apply the inductive hypothesis on the balls $B_j$, and the remainder of the inductive step follows in a routine way once we observe one further crucial geometric fact, this time regarding the transversal tubes:

\begin{lemma} \label{lem: transverse}
Each tube $T\in\mathbb{T}$ belongs to $\mathbb{T}_{j,trans}$ for at most $D^{O(1)}=R^{O(\delta_{deg})}$ values of $j$.
\end{lemma}

We can think of this as analogous to the fact that any line intersects $Z(P)$ at most $D$ times, except we must now replace $D$ with $D^{O(1)}$ in the setting where we are now considering intersections of tubes and balls in a neighborhood of $Z(P)$. Of course, the ``transversality" assumption is essential to this observation, as we could otherwise have tubes parallel to $Z(P)$ which intersect many $B_j$.
This estimate will be discussed further in Section \ref{sec:geometry}.

\subsection{Structure of this study guide}
In this introduction, we motivated and outlined the core argument. The remainder of these notes will elaborate on certain features of \cite{restest}.

Section~\ref{sec: counterexample} contains an outline of Guth's first example showing that the exponent $p=3.25$ is sharp for the method (from Section 0.3 in the original paper), but we fill in many of the probablistic methods, which may be helpful to those less familiar with the standard arguments.

Section \ref{sec: Induction} gives further details on the proof of Theorem \ref{induction}, emphasizing the more difficult steps of the argument for the benefit of readers less familiar with the techniques.

Section \ref{sec:geometry} contains a discussion of the geometric input to \cite{restest}. We give an outline of the proof of the transversal estimate Lemma \ref{lem: transverse} and a detailed proof of the tangential estimate Lemma \ref{lem: tangent}. Guth's paper \cite{restest} opened avenues to the latest improvements on restriction, and these are discussed at the end. 

\section{Counterexample} \label{sec: counterexample}  

Fix $K>1$ and $\epsilon >0$. As in the introduction, partition the paraboloid (truncated, graphed over the unit ball $B_2(0,1)$) into $K^{-1}$-caps $\tau$, and define the broad part 
\[\text{Br}_{K^{-\epsilon}} Ef(x) = \chi_{\text{$K^{-\epsilon}$-broad}}(x) |Ef(x)|, \]
where $x$ is $K^{-\epsilon}$-broad if for any $K^{-1}$-cap $\tau$, 
\[ |Ef_{\tau}(x) | \leq K^{-\epsilon} |Ef(x)|, \]
and $\chi_{\text{$K^{-\epsilon}$-broad}}$ is the indicator function of the broad points. 

\begin{theorem}
Let $p \in [3, \infty)$. Suppose that for all $\epsilon >0$, there exists a constant $K=K(\epsilon) > 1$ such that for all $R \geq 1$, 
\begin{equation} \label{broadinequality} \left\lVert \text{Br}_{K^{-\epsilon} } Ef \right\rVert_{L^p(B_R)} \leq C_{\epsilon} R^{\epsilon} \|f\|_{2}^{12/13} \|f\|_{\infty}^{1/13}, \end{equation}
with $\lim_{\epsilon \to 0} K(\epsilon) = \infty$. Then $p \geq 13/4$.
\end{theorem}

Observe that this inequality is the broad part estimate given in the introduction. Thus this theorem essentially says that $p=3.25$ is the best possible exponent for this method, without further refinement.

\begin{proof} Given $\epsilon >0$, choose $\delta>0$ with $\delta \ll \epsilon$ (we take $\delta=\epsilon^2$ and $\epsilon < 1/100$, for concreteness). Cover the paraboloid by $R^{-1/2}$-caps $\theta$. Let $\Lambda$ be the set of caps intersecting the $R^{-1/2}$-neighbourhood of $\{ (x_1,x_2, |x|^2) \in \mathbb{P}: x_1 = 0 \}$. Then $|\Lambda| \sim R^{1/2}$. 

Fix an integer $B$ with $K R^{\delta} \leq B \leq R^{1/2}$ to be chosen later. Let $S$ be the slab $S := [-R^{1/2}, R^{1/2}] \times [-R,R] \times [-R,R]$. Let $\mathbb{T} = \bigcup_{\theta \in \Lambda} \mathbb{T}_{\theta}$, where each $\mathbb{T}_{\theta}$ is a finitely overlapping set of $R^{1/2} \times R^{1/2} \times R$-tubes parallel to the normal at $\theta$, which cover the slab $S$. Let $\mathcal{S}$ be the (finite) set of subsets $\mathbb{W}$ of $\mathbb{T}$ such that $|\mathbb{W}\cap \mathbb{T}_{\theta}| = B$ for every $\theta \in \Lambda$. Let $\mathbb{P}$ be the uniform probability measure on $\mathcal{S}$ which gives each element equal probability. Then for each $x \in S$ and $\theta \in \Lambda$,  
\[ \mathbb{P}\left\{ x \in T \text{ for some } T \in \mathbb{W} \cap \mathbb{T}_{\theta} \right\} \sim B/ R^{1/2}. \]
It follows that for any $x \in S$,
\[ \mathbb{E}\left(\sum_{T \in \mathbb{W} }\chi_T(x) \right) =   \sum_{\theta}\mathbb{E}\left(  \sum_{T \in \mathbb{W} \cap \mathbb{T}_{\theta} } \chi_T(x) \right) \sim B. \]
It follows that
\begin{equation} \label{Wcond} \int_S \sum_{T \in \mathbb{W}} \chi_T \sim m(S) B, \end{equation}
for all $\mathbb{W}$ in a set of probability $\gtrsim 1$, where $m(S)$ is the Lebesgue measure of $S$. Cover $S$ with cubes $Q$ of side length $R^{1/2}$. For each cube $Q$ and each $K^{-1}$-cap $\tau$, we claim that
\begin{equation}  \label{claimedineq} \mathbb{P}\left\{ \sup_{x \in Q} \sum_{T \in \mathbb{W} \cap \mathbb{T}_{\tau}} \chi_{T}(x) \geq CBK^{-1} \right\} \leq e^{-(C/2) B K^{-1} }, \end{equation}
if $C>0$ is a sufficiently large absolute constant, where $\mathbb{T}_{\tau}$ is the union over $\mathbb{T}_{\theta}$ with $\theta \cap \tau \neq \emptyset$. Roughly speaking, this is because we are  summing $\lesssim R^{1/2} K^{-1}$ identically independent distributed Bernoulli random variables, each of which takes value 1 with probability $\sim B/ R^{1/2}$, and the probability that the sum is very far from the expected value is exponentially small. This is a special instance of the fact that if we conduct $N$ flips of a biased coin that shows heads with probability $p$, then the probability that the fraction of heads is very far from $p$ is exponentially small in $N$. This probability heuristic arises often in work on the Kakeya problem and the restriction conjecture. 

To make the above heuristic precise, for each $Q$ and $\tau$ let
\[ \mathbb{T}_{Q, \tau} = \{ T \in \mathbb{T}: T \cap Q \neq \emptyset: T \in \mathbb{T}_{\theta} \text{ with } \theta \cap \tau \neq \emptyset \}. \]
For $T \in \mathbb{T}_{Q, \tau}$ let $Y_T = 1$ if $T \in \mathbb{W}$, and zero otherwise. Then the probability above is bounded by  
\[ \mathbb{P}\left\{  \sum_{T \in \mathbb{T}_{Q, \tau}} Y_T \geq CBK^{-1} \right\}. \]
To ensure the random variables are actually independent, for each $\theta$ we need to pick exactly one tube $T \in \mathbb{T}_{Q, \tau} \cap \mathbb{T}_{\theta}$ from the $\sim 1$ tubes in $\mathbb{T}_{Q, \tau} \cap \mathbb{T}_{\theta}$, but this would only change the $CBK^{-1}$ by a harmless constant factor, so to simplify notation we ignore this technicality below.  

By exponentiating and then using Chebychev's inequality, the probability that any individual sum exceeds $CBK^{-1}$ is
\begin{align*} \mathbb{P}\left\{ e^{\sum_{T \in \mathbb{T}_{Q, \tau}} Y_T} \geq e^{CBK^{-1}}\right\} &\leq e^{-CBK^{-1}} \int e^{\sum_{T \in \mathbb{T}_{Q, \tau}} Y_T}  \, dP \\
&= e^{-CBK^{-1}} \prod_{T \in \mathbb{T}_{Q, \tau}} \int e^{Y_T} \, dP \\
&\leq e^{-CBK^{-1}} \prod_{T \in \mathbb{T}_{Q, \tau}} \left( 1 + 1000BR^{-1/2} \right) \end{align*} 
If we apply the inequality $\log(1+x) \leq x$ to the right-hand side, we get
\[\log\left( \mathbb{P}\left\{ \sum_{T \in \mathbb{T}_{Q, \tau}} Y_T \geq CBK^{-1} \right\} \right) \\
\leq -CBK^{-1} + \sum_{T \in \mathbb{T}_{Q, \tau}} O\left( BR^{-1/2}\right).\]
The sum has $\lesssim R^{1/2}K^{-1}$ terms and is therefore dominated by the first term (provided $C$ is now chosen sufficiently large), so this gives 
\[ \mathbb{P}\left\{ \sup_{x \in Q} \sum_{T \in \mathbb{W} \cap \mathbb{T}_{\tau}} \chi_{T}(x) \geq CBK^{-1} \right\} \leq e^{-(C/2) B K^{-1} }. \]
This verifies the claimed inequality \eqref{claimedineq}. By summing over the cubes $Q$ and the caps $\tau$, and using the trivial union bound, it follows that the probability that some point in the slab has at least $CBK^{-1}$ tubes $T  \in \mathbb{W}$ passing through it corresponding to a single $\tau$ is 
\[ \lesssim R K e^{-(C/2) B K^{-1} }. \]
Since $B \geq KR^{\delta}$, this is exponentially small in $R$. It follows that, if $R$ is sufficiently large (depending on $\epsilon$ and $\delta$), then for a set of $\mathbb{W}$ of probability $\gtrsim 1$, no point in the slab has $\geq CBK^{-1}$ tubes passing through it corresponding to a single $\tau$, and \eqref{Wcond} holds. By the pigeonhole principle, such sets $\mathbb{W}$ also have the property that no point in the slab has $\geq 100CB$ tubes passing through it. From \eqref{Wcond} and the preceding discussion, for some absolute constant $c>0$ we can find a specific set $\mathbb{W}$ (non-random) such that
\[ m(F) \sim m(S) \sim R^{5/2}, \qquad F = \left\{ x \in S: \sum_{T \in \mathbb{W}} \chi_T \geq c B\right\}, \]
such that no point in the slab has $\geq CBK^{-1}$ tubes passing through it corresponding to a single $\tau$. Let $\varepsilon = (\varepsilon_T)_{T \in \mathbb{W}}$ be a sequence of independent and identically distributed random variables, taking the values $\pm 1$ with equal probability, on the same probability space. For each $\theta$, let $\phi_{\theta}$ be a smooth bump function supported on a $R^{-1/2} \times R^{-1/2} \times R^{-1}$ rectangular box centred at the centre of $\theta$, and $\sim 1$ on a slightly smaller box, where 
\[ |E\phi_{\theta}(x)| \gtrsim R^{-1}, \]
on a $\sim R^{1/2} \times R^{1/2} \times R$ tube $T_{\theta,0}$ centered at the origin and dual to $\theta$ (meaning that the long side of $T_{\theta,0}$ has the same direction as the short side of $\theta$). If we choose the implicit constants defining the support of $\phi_{\theta}$ small enough, then the set $\mathbb{T}_{\theta}$ used above will be a boundedly overlapping cover of the slab $S = [-R^{1/2}, R^{1/2}] \times [-R,R] \times [-R,R]$ by translates $T$ of $T_{\theta,0}$. Define 
\[ \phi_T(\omega) = \phi_{\theta}(\omega) e^{-2\pi i \langle \omega,x_T\rangle}, \]
with $x_T \in \mathbb{R}^3$ chosen such that $T_0+x_T = T$, so that
\[ E\phi_T(x) = E\phi_{\theta}(x-x_T), \] has modulus $\gtrsim R^{-1}$ on $T$.  Define 
\[ f = f_{\varepsilon} = \sum_{T \in \mathbb{W}} \varepsilon_T \phi_T. \] Then, for each $x \in F$, Khintchine's inequality gives
\[ \left(\mathbb{E}  \left\lvert Ef(x) \right\rvert^p \right)^{1/p} \sim \left(\sum_{T \in \mathbb{W}} |E\phi_T(x)|^2 \right)^{1/2} \gtrsim R^{-1} B^{1/2}. \]
We note that the set $\mathbb{W}$ used above is fixed (non-random), and the expectation is for the random variables $\varepsilon_T$. It follows that 
\[ \mathbb{E}\left(\int_{F} |Ef|^p \right)\gtrsim m(F) R^{-p} B^{p/2} \sim R^{\frac{5}{2}-p } B^{\frac{p}{2}}. \] 
We want to replace the integral on the left-hand side of the above with the broad norm, so we will show that the contribution of the narrow points to the above integral is much smaller than the right-hand side of the above. If $x$ is $K^{-\epsilon}$-narrow for $Ef$, then 
\[ |Ef(x)|^p \leq K^{\epsilon p} \sum_{\tau} |Ef_{\tau}|^p. \]
Thus (by Khintchine's inequality again)
\begin{align*} \mathbb{E}\left(\int_{\text{narrow} \cap S} |Ef|^p \right) &\lesssim K^{\epsilon p}\sum_{\tau} \int_{S} \mathbb{E} |Ef_{\tau}|^p \\
&\lesssim K^{\epsilon p}\sum_{\tau} \int_{S} \left(\sum_{T \in \mathbb{W} \cap \mathbb{T}_{\tau} } |E\chi_{\tau}\phi_T(x)|^2 \right)^{p/2} \\
&\lesssim K^{1+\epsilon p -\frac{p}{2} } m(S) R^{-p}B^{\frac{p}{2}} \\
&\lesssim K^{1+\epsilon p - \frac{p}{2}} R^{ \frac{5}{2} - p} B^{\frac{p}{2}}. \end{align*}
(The fact that the functions $E\chi_{\tau}\phi_T$ are not literally supported on $T$ is a minor technicality which can morally be ignored in the above. To get around it, we need to sum a geometric series over dyadic numbers $M$ and use that for $x$ outside $MT$, $|E\chi_{\tau}\phi_T(x)|$ is $\leq C_N M^{-N}$, and that for a given $M$, by pigeonholing there cannot be more than $M^{O(1)}B K^{-1}$ tubes $MT$ passing through $x$ corresponding to a single $\tau$.)
Since $p>2$ and $\lim_{\epsilon \to 0} K(\epsilon)= \infty$, the above is much smaller than the lower bound $R^{\frac{5}{2}-p } B^{p/2}$ for the integral $\mathbb{E}\left(\int_{F} |Ef|^p \right)$ above, provided $\epsilon$ is sufficiently small. Since the broad and narrow points partition the slab, it follows that  
\[ \mathbb{E}\left(\int_{\text{broad} \cap S} |Ef|^p \right) \gtrsim R^{\frac{5}{2}-p } B^{p/2}. \]

By Khintchine's inequality (or just $L^2$ orthogonality), 
\[ \mathbb{E} \|f\|_{2}^2 \sim B R^{-1/2}. \]
It follows that we can find a single function $f$ (non-random) such that
\[ \left(\int_{{\text{broad} \cap S}} |Ef|^p\right)^{1/p} \gtrsim R^{\frac{5}{2p}-1 } B^{1/2} \quad \text{ and } \quad  \|f\|_{2} \lesssim B^{1/2} R^{-1/4}. \]
The triangle inequality gives
\[ \|f\|_{\infty} \lesssim B. \]
Thus if \eqref{broadinequality} holds, then 
\[ R^{\frac{5}{2p}-1 } B^{\frac{1}{2}} \lesssim B^{\frac{6}{13}} R^{\frac{-3}{13}} B^{\frac{1}{13}} R^{\epsilon}. \]
Simplifying gives 
\[ R^{\frac{5}{2p}} \lesssim B^{\frac{7}{13} - \frac{1}{2}} R^{ \frac{10}{13} + \epsilon}. \]
Since the exponent of $B$ is positive, we get the strongest restriction on $p$ by taking $B$ as small as possible, namely $B = K R^{\delta}$. Thus 
\[  R^{\frac{5}{2p}} \lesssim K^{\frac{7}{13} - \frac{1}{2}}   R^{\frac{10}{13} +\delta+ \epsilon}. \]
Since $K$ is constant (for fixed $\epsilon$), sending $R \to \infty$ gives 
\[ \frac{5}{2p} \leq \delta +\epsilon+ \frac{10}{13}. \]
Since this holds for any $0 < \delta \ll \epsilon$ and $\epsilon>0$, we get 
\[ \frac{5}{2p} \leq \frac{10}{13}, \]
and rearranging this gives $p \geq \frac{13}{4}$.  \end{proof}

The example above seems to really require the assumption that $K \to \infty$ as $\epsilon \to 0$, and this is not mentioned in Guth's paper, so it may be asked whether one could get a better $p$ by avoiding this. However, if $K$ were to remain bounded as $\epsilon \to 0$, then $K^{-\epsilon}$ would be greater than 1/2 for $\epsilon$ sufficiently small, and in this situation the degree 2 algebraic surface example from Guth's paper also gives the restriction $p \geq 13/4$. In his paper, he writes ``Because 1/2 is larger than $K^{-\epsilon}$, this example is not directly relevant to Theorem 0.3, but I think it is morally relevant''. However, the above working actually shows that the example can be used to remove the assumption $K(\epsilon) \to \infty$ as $\epsilon \to 0$, so it seems to be directly relevant.

%\subsection{Comments etc on counterexample}

% \section{White Lie Harmonic Analysis Argument}
% Should assume that tails don't exist and wave packets are orthogonal. Do just the induction on balls: makes cellular and broad case look similar...
% Broad norm introduced solely for case when tangential part dominates...make this clear. In cellular part make clear why $3/2$ is important. 

\section{Details on the Proof of Theorem \ref{induction}} \label{sec: Induction}
As detailed in the introduction, the main theorem reduces to Theorem \ref{induction}, which implies the desired broad estimate Theorem \ref{broadest}. Here, we shall illustrate how this happens and then elaborate the inductive procedure. For convenience, we reiterate Theorem \ref{induction} here:

\begin{theorem}\label{3.1}
    For any $\varepsilon>0$, there exists $K,L$ and a small $\delta_{trans}\in(0,\varepsilon)$, depending only on $\varepsilon$, so that the following holds.

    Suppose that $S$ is a truncated paraboloid, that the caps $\tau$ cover $S$ with multiplicity at most $\mu$, and that $\alpha\geq K^{-\varepsilon}.$

    If for any $\tau$ and any $\omega\in S$, we have
    \begin{equation}\label{3.1.1}
        \oint_{B(\omega,R^{-1/2})\cap S}|f_\tau|^2\leq1,
    \end{equation}
    then
    \begin{equation}\label{3.1.2}
        \int_{B_R}Br_\alpha Ef^{3.25}\leq C_\varepsilon R^\varepsilon(\displaystyle\sum_\tau\int_S|f_\tau|^2)^{(3/2)+\varepsilon}R^{\delta_{trans}\log(K^\varepsilon\alpha\mu)}.
    \end{equation}
    Moreover, $\displaystyle\lim_{\varepsilon\longrightarrow0^+}K(\varepsilon)=\infty$.
\end{theorem}

This implies Theorem \ref{broadest} as follows. We first note that the desired inequality
$$\|\text{Br}_{K^{-\varepsilon}} Ef\|_{L^{3.25}(B_R)}\leq C_\varepsilon R^\varepsilon \|f\|_2^{12/13}\|f\|_\infty^{1/13}$$
is preserved by scaling, so that we may assume $\|f\|_\infty=1$, which then satifies (\ref{3.1.1}) and allows us to apply Theorem \ref{3.1}. Then, we obtain
$$\begin{array}{rcl}
    \|Br_{K^{-\varepsilon}}Ef\|_{L^{3.25}(B_R)} & \lesssim & R^{O(\varepsilon)+\delta_{trans}\log(K^\varepsilon\alpha\mu)}\|f\|_{L^2(S)}^{12/13+O(\varepsilon)} \\
    & \leq & R^{O(\varepsilon+\delta_{trans}\log(K^\varepsilon\alpha\mu)}\|f\|_{L^2(S)}^{12/13}|S|^{O(\varepsilon)}\|f\|_{L^\infty}^{O(\varepsilon)} \\
    & \lesssim & R^{O(\varepsilon)}\|f\|_{L^2(S)}^{12/13}\|f\|_{L^\infty}^{1/13},
\end{array}$$
where the first inequality is due to Theorem \ref{3.1}, and the last inequality is obtained by choosing $\delta_{trans}$ and $K$ carefully and the fact that $\|f\|_{L^\infty}$ is normalized to $1$ and that $|S|$ is a constant.

We now proceed with the proof of Theorem \ref{3.1}. The key structure is an induction on $R$ and $\displaystyle\sum_\tau\int_S|f_\tau|^2$.

\subsection{Base case}

\begin{lemma}\label{R=1}
    With the same assumptions in Theorem \ref{3.1}, (\ref{3.1.1}) holds for $R=R_0$ for some large $R_0$.
\end{lemma}
\begin{proof}
With $\varepsilon>0$ fixed and $K=K(\varepsilon)$ to be determined later, the base case is when $R_0=R(\varepsilon)$ such that $K\leq R_0^C$ for some constant $C$ or
\begin{equation}\label{basecase}
    \displaystyle\sum_\tau\int_S|f_\tau|^2\leq R^{-1000}.
\end{equation}

If $R=R_0\leq1$, we have
$$\begin{array}{rcl}
    \int_{B_1}Br_\alpha Ef^{3.25} & \lesssim & \|f\|_{L^1}^{3.25} \\
     & \leq & \|f\|_{L^2}^{3.25} \\
     & = & \|f\|_{L^2}^{3+2\varepsilon}\|f\|_{L^2}^{3.25-3-2\varepsilon} \\
     & \leq & \|f\|_{L^2}^{3+2\varepsilon} \\
     & \lesssim & (\displaystyle\sum_\tau\int_S|f_\tau|^2)^{3/2+\varepsilon}.
\end{array}$$
The second step is by H\"older's inequality, and the third inequality is due to the assumption (\ref{3.1.1}).
\end{proof}

\begin{remark}
Guth proves the base case $R=1$. This is equivalent to the above argument after scaling.
\end{remark}

Now we suppose (\ref{basecase}) is true.

\begin{lemma}\label{smallsum}
If (\ref{basecase}) is true, then (\ref{3.1.2}) is true.
\end{lemma}
\begin{proof}
\begin{equation}\label{base1}
    \int_{B_R}Br_\alpha Ef^{3.25}\lesssim R^3(\displaystyle\sum_\tau\int_S|f_\tau|)^{3.25}
\end{equation}
because $|B_R|\sim R^3$ and
$$|Br_\alpha Ef|\leq\int_S|f|\leq\displaystyle\sum_\tau\int_S|f_\tau|.$$
The last term can be estimated by
\begin{equation}\label{base2}
    \begin{array}{rcl}
    \displaystyle\sum_\tau\int_S|f_\tau| & \leq & \displaystyle\sum_\tau|\tau|^{1/2}(\int_S|f_\tau|^2)^{1/2} \\
    & \leq & (\displaystyle\sum_\tau|\tau|)^{1/2}(\displaystyle\sum_\tau\int_S|f_\tau|^2)^{1/2} \\
    & \sim & (\mu K^{-1})(\displaystyle\sum_\tau\int_S|f_\tau|^2)^{1/2},
\end{array}
\end{equation}
where the first inequality is by H\"older's inequality, the second is Cauchy-Schwarz inequality, and the last line is due to the estimate of number of caps $\tau$ comparable to $ O(\mu)$ and the radius of $\tau$ bounded by $\mu^{1/2}K^{-1}$. Since $K$ is chosen upon $\varepsilon$. One can choose such that $\mu K^{-1}\leq R_0^{ O(\varepsilon)}\leq R^{ O(\varepsilon)}$. Then (\ref{base1}), (\ref{base2}) and (\ref{basecase}) gives the estimate
$$\int_{B_R}|Br_\alpha Ef|^{3.25}\lesssim R^4(\displaystyle\sum_\tau\int_S|f_\tau|^2)^{13/8}\lesssim R^{-100}(\displaystyle\sum_\tau\int_S|f_\tau|^2)^{3/2+\varepsilon}.$$
Here $-100$ is much smaller than the exponent of $R$ in (\ref{3.1.2}). Thus, the base case is done.
\end{proof}

\subsection{Induction step}
Now let's assume (\ref{3.1.2}) is true for some large $R$ or $\displaystyle\sum_\tau\int_S|f_\tau|^2$. It suffices to prove that (\ref{3.1.2}) is also true for $2R$ or $\displaystyle\sum_\tau\int_S|f_{\tau|,new}^2\leq2\displaystyle\sum_\tau\int_S|f_\tau|^2$.

Guth \cite{restest} estimates the left hand side of (\ref{3.1.2}) $\int_{B_R}|Br_\alpha Ef|^{3.25}$ on the physical side by summing up the wave packets (tubes) corresponding to each cap $\tau$. Furthermore, by polynomial partitioning, there exists a polynomial $P$ of degree at most $D$ such that the zero set $Z(P)$ divides $\R^3$ into $\sim D^3$ components $O_i$ that equally divide $\int_{B_R}|Br_\alpha Ef|^{3.25}$, i.e.,
$$\int_{ O_i\cap B_R}|Br_\alpha Ef|^{3.25}\sim D^{-3}\int_{B_R}|Br_\alpha Ef|^{3.25}.$$

Then the tubes can be split into three cases: tubes intersecting with some cell $O_i$, those intersecting $Z(P)$ transversely, i.e., the direction of the tube lies far away from the tangent space of $Z(P)$ at the intersection, and those intersecting $Z(P)$ tangentially, i.e., the direction of the tube lies close or in the tangent space of $Z(P)$. Since tubes are not lines, but have thickness, we may consider a neighborhood $W$ of $Z(P)$ of radius equal to that of a tube and $O'_i=O_i\setminus W$ instead of $Z(P),O_i$. It is also worth noting that the said three cases are not necessarily mutually exclusive, but this does not affect the estimate.

\subsubsection{Cellular case}
Suppose $W$ contributes less than $O_i$ in $\int_{B_R}|Br_\alpha Ef|^{3.25}$. Since the integral in each cell contributes equally, we may estimate $\int_{O'_i\cap B_R}|Br_\alpha Ef|^{3.25}\sim D^{-3}\int_{B_R}|Br_\alpha Ef|^{3.25}$ for any $i$.

\begin{lemma}\label{2times}
Suppose $f$ satisfies (\ref{3.1.1}), then $f_i$ satisfies (\ref{3.1.1}).
\end{lemma}

\begin{proof}
Let $\omega\in S$. Then by the wave packet decomposition, we have
$$\oint_{B(\omega,R^{-1/2})\cap S}|f_{\tau,i}|^2\lesssim\oint_{B(\omega,10R^{-1/2})\cap S}|f_\tau|^2\lesssim1,$$
where $f_{\tau,i}=\displaystyle\sum_{T\in\T_\tau\cap\T_i}f_{\tau,T}$. Hence $f_i$ satisfies the assumption (\ref{3.1.1}).

Moreover, we may observe that $\displaystyle\sum_\tau\int_S|Cf_{\tau,i}|^2\leq\frac{1}{2}\displaystyle\sum_\tau\int_S|f_{\tau}|^2$ for some constant $C\in[1/2,1]$.
\end{proof}

\begin{lemma}\label{cell}
If $x\in O'_i$ and $R$ is large enough, then
\begin{equation}\label{cellind}
    Br_\alpha Ef(x)\leq2Br_{2\alpha}Ef_i(x)+Err.
\end{equation}
The error term $Err\lesssim R^{-900}\displaystyle\sum_\tau\|f_\tau\|_2$.
\end{lemma}
\begin{remark}
$Err$ arises from wave packet decomposition and can be easily dominated.
\end{remark}
Let's first prove (\ref{3.1.2}) by Lemma \ref{cell} when the cellular part is dominating.

\begin{prop}\label{celltothm}
Suppose satisfy the assumption of Theorem \ref{3.1} holds for some large enough $R$ and $f$, and that (\ref{3.1.2}) holds for $R$ and $f_i$ if (\ref{3.1.1}) is satisfied, then (\ref{3.1.2}) holds for the same $R$ and $f$.
\end{prop}
\begin{remark}
If (\ref{3.1.2}) does not hold for $f_i$, we may apply polynomial partitioning on $f_i$ with the same degree $D$ to obtain a smaller $\displaystyle\sum_\tau\int_S|f_\tau|^2$. Since the shrink is by a factor of $1/2$ after each partitioning, it will eventually drop below $R^{-1000}$, reducing the problem to the base case.
\end{remark}

\begin{proof}
$$\int_{B_R}Br_\alpha Ef^{3.25}\lesssim D^3\int_{B_R\cap O'_i}Br_\alpha Ef^{3.25}\lesssim D^3\int_{B_R}|Br_{2\alpha}Ef_i|^{3.25}+Err,$$
where the first inequality is by the polynomial partitioning, the second by Lemma \ref{cell} with the error term $Err\sim R^{-1000}(\displaystyle\sum_{\tau}\|f_\tau\|_2)^{3.25}$. By Cauchy-Schwarz inequality, we have $Err\lesssim R^{-1000}\mu^{13/8}(\displaystyle\sum_\tau\int_S|f_\tau|^2)^{(3/2)+(1/13)}$, which is much smaller than the right hand side of (\ref{3.1.2}) because by the assumption (\ref{3.1.1}), we have $\displaystyle\sum_\tau\int_S|f_\tau|^2\leq1$ and $\varepsilon\leq1/13$.

Now it suffices to show that $D^3\int_{B_R}|Br_{2\alpha}Ef_i|^{3.25}$ is also bounded above by the right hand side of (\ref{3.1.2}).

By simple algebra for the central axis of each tube, each $T$ intersects at most $(D+1)$ of the reduced cells $O'_i$, so $\displaystyle\sum_i\int|f_{\tau,i}|^2\lesssim D\int|f_\tau|^2$, which implies that there exists some $i$ such that $\int|f_{\tau,i}|^2\lesssim D^{-2}\int|f_\tau|^2$ because the number of cells is $\sim D^3$. Hence
\begin{equation}\label{fitof}
    \displaystyle\sum_\tau\int|f_{\tau,i}|^2\lesssim D^{-2}\displaystyle\sum_\tau\int|f_\tau|^2.
\end{equation}
As mentioned above, the choice of $i$ does not affect the estimate of $\int_{B_R\cap O'_i}Br_\alpha Ef_i^{3.25}$, so let's fix this $i$.

By Lemma \ref{celltothm}, (\ref{3.1.1}) holds for $f_i$ and $R$. By the induction assumption, (\ref{3.1.2}) holds for $f_i$ and $R$. Then we have
\begin{equation}\label{fitofind}
    \int_{B_R}Br_\alpha Ef^{3.25}\lesssim C_\varepsilon R^\varepsilon R^{\delta_{trans}\log(K^\varepsilon2\alpha\mu)}(\displaystyle\sum_\tau\int|f_{\tau,i}|^2)^{(3/2)+\varepsilon}.
\end{equation}
By (\ref{fitofind}) and (\ref{fitof}), we obtain
$$\int_{B_R}Br_\alpha Ef^{3.25}\lesssim (D^{-2\varepsilon}R^{ O(\delta_{trans})})C_\varepsilon R^\varepsilon R^{\delta_{trans}\log(K^\varepsilon\alpha\mu)}(\displaystyle\sum_\tau\int|f_\tau|^2)^{(3/2)+\varepsilon}.$$

By choosing $D,\delta_{trans}$ carefully, we may obtain a constant bound for the bracketed term. Thus, the induction is closed.
\end{proof}

\begin{remark}
Guth \cite{Guth_polypart_2} provides an alternative proof of the cellular case by induction on $R$.
\end{remark}

Now let's turn to the proof of Lemma \ref{cell}.

\begin{proof}[Proof of Lemma \ref{cell}]
We may assume $|Br_\alpha Ef(x)|\geq R^{-1000}(\displaystyle\sum_\tau\|f_\tau\|_2)$ and that $x$ is $\alpha$-broad for $Ef$. Otherwise, the the inequality trivially holds. Now it suffices to prove that $x$ is also $2\alpha$-broad $Ef_i$, i.e., for any cap $\tau$, we need $|Ef_{\tau,i}(x)|\leq2\alpha|Ef_i(x)|$.

By the definition of broad points, we have $|Br_\alpha Ef(x)|\leq|Br_{2\alpha}Ef(x)|$.
By wave packet decomposition, for any $x\in O'_i$, we have
\begin{equation}\label{cell1}
    \begin{array}{rcl}
    Ef_\tau(x) & = & \displaystyle\sum_{T\in\T}Ef_{\tau,T}(x)+ O(R^{-1000}\|f_\tau\|_2) \\
    & = & Ef_{\tau,i}(x)+ O(R^{-990}\|f_\tau\|_2).
\end{array}
\end{equation}
The second equality is due to the fact that when a tube $T$ is away from $ O'_i$, $|Ef_{\tau,T}(x)|\leq R^{-1000}\|f_\tau\|_2$. Furthermore, the number of such tubes is at most the number of tubes corresponding to $\tau$, which is loosely bounded by $R^{10}$.
Summing over $\tau$, we obtain
\begin{equation}\label{cell2}
    Ef(x)=Ef_i(x)+ O(R^{-990}\displaystyle\sum_\tau\|f_\tau\|_2).
\end{equation}
Hence
\begin{equation}\label{cell3}
    |Ef_i(x)|=|Ef(x)|- O(R^{-990}\displaystyle\sum_\tau\|f_\tau\|_2)\geq\frac{1}{2}R^{-900}\displaystyle\sum_\tau\|f_\tau\|_2.
\end{equation}
Therefore,
\begin{equation*}
    \begin{array}{rcl}
    |Ef_{\tau,i}(x)| & \leq & |Ef_\tau(x)|+ O(R^{-990}\|f_\tau\|_2) \\
    & \leq & \alpha|Ef(x)|+ O(R^{-990}\|f_\tau\|_2) \\
    & \leq & \alpha|Ef_i(x)|+ O(R^{-990}\displaystyle\sum_\tau\|f_\tau\|_2) \\
    & \leq & 2\alpha|Ef_i(x)|.
\end{array}
\end{equation*}
The first and third inequalities are due to (\ref{cell1}) and (\ref{cell2}), the second due to the assumption that $x$ is $\alpha$-broad for $f$, and the last one due to (\ref{cell3})
\end{proof}

\subsubsection{Transverse case}
Then let's consider the case when $\int_{W\cap B_R}Br_\alpha Ef^{3.25}$ dominates. In this case, we are proving by induction on $R$. We first cover $W$ with balls $B_j$ of radius $R^{1-\delta}$ for some small $\delta<\varepsilon$. Let's also recall the following two notations
$$f_{j,trans}\coloneqq\displaystyle\sum_\tau\displaystyle\sum_{T\in\T_{j,trans}}f_{\tau,T},\quad f_{j,tang}\coloneqq\displaystyle\sum_\tau\displaystyle\sum_{T\in\T_{j,tang}}f_{\tau,T},$$
where $\T_{j,trans}$ is the set of tubes that intersect $W$ via $B_j$ transversely and $\T_{j,tang}$ tangentially. Then we may define
$$f_{\tau,j,trans}=\displaystyle\sum_{T\in\T_{j,trans}}f_{\tau,T},\quad f_{I,j,trans}\coloneqq\displaystyle\sum_{\tau\in I}f_{\tau,j,trans},$$
where $I$ is an arbitrary subset of the set of caps $\tau$. 

\begin{lemma}\label{zerodom}
If $x\in B_j\cap W$ and $\alpha\mu\leq10^{-5}$, then
\begin{equation}\label{zerodomineq}
    Br_\alpha|Ef(x)|\leq2(\displaystyle\sum_IBr_{2\alpha}|Ef_{I,j,trans}(x)|+K^{100}Bil(Ef_{j,tang})(x)+Err),
\end{equation}
where
$$Bil(Ef_{j,tang})\coloneqq\displaystyle\sum_{\tau_1,\tau_2\text{ non-adjacent}}|Ef_{\tau_1,j,tang}|^{1/2}|Ef_{\tau_2,j,tang}|^{(1/2)}.$$
Two capt $\tau_2,\tau_2$ are said to be non-adjacent if their distance is $\geq K^{-1}$.

Similarly, the error term $Err\coloneqq O(R^{-900}\displaystyle\sum_\tau\|f_\tau\|_2^)$.
\end{lemma}

In this section, we are dealing with the case when the first term on the right hand side of (\ref{zerodomineq}) is dominating.

First, let's introduce a combinatorial lemma that will be useful.

\begin{lemma}\label{transct}
Each tube $T\in\T$ belongs to at most $\text{Poly}(D)=R^{ O(\delta_{deg})}$ different sets $\T_{j,trans}$.
\end{lemma}
\begin{remark}
More details about Lemma \ref{transct} can be found in Section \ref{sec:geometry}.
\end{remark}

\begin{lemma}
Suppose $W$ dominates in $\int_{B_R}Br_\alpha Ef^{3.25}$, that the transverse term or the error term dominates in (\ref{zerodomineq}), and that Theorem \ref{3.1} is true for $R^{1-\delta}$, then Theorem \ref{3.1} is also true for $R$.
\end{lemma}

\begin{proof}
Suppose $f$ satisfies (\ref{3.1.1}). By the assumption and Lemma \ref{zerodom}, we have
$$\begin{array}{rcl}
    \int_{B_R}Br_\alpha Ef^{3.25} & \leq & \displaystyle\sum_j\int_{B_j}Br_\alpha Ef^{3.25} \\
    & \lesssim & \displaystyle\sum_{j,I}\int_{B_j}Br_{2\alpha}\int_{B_j}Br_{2\alpha}|Ef_{I,j,trans}(x)| \\
    & & +\displaystyle\sum_jK^{100}\int_{B_j}Bil(Ef_{j,tang})(x)+Err, \\
\end{array}$$
where $Err\sim R^{-1000}(\displaystyle\sum_\tau\|f_\tau\|_2)^{3.25}$. By the same argument as in Proposition \ref{celltothm}, if the error term dominates, (\ref{3.1.2}) trivially holds. It remains to prove the case when the first term dominates, i.e.,
$$\int_{B_R}Br_\alpha Ef^{3.25}\lesssim\displaystyle\sum_{j,I}\int_{B_j}Br_{2\alpha}Ef_{I,j,trans}^{3.25}.$$

By the wave packet decomposition,
$$\begin{array}{rl}
     & \oint_{B(\omega,R^{-(1/2)(1-\delta)})\cap S}|f_{\tau,j,trans}|^2 \\
    \lesssim & \displaystyle\sum_{\omega^\prime}\int_{B(\omega^\prime,R^{-\frac{1}{2}}\cap S)}|f_{j,trans,I,\tau}|^2 \\
    \sim & \displaystyle\sum_{\omega^\prime}R^{-1}\oint_{B(\omega^\prime,R^{-\frac{1}{2}}\cap S)}|f_{j,trans,I,\tau}|^2 \\
    \lesssim & \displaystyle\sum_{\omega^\prime}R^{-1}\oint_{B(\omega^\prime,10R^{-\frac{1}{2}}\cap S)}|f_\tau|^2 \\
    \lesssim & 1,
\end{array}$$
where $\omega^\prime$ are the centers of balls of radius $R^{-\frac{1}{2}}$ that cover $B(\omega,R^{-(1/2)(1-\delta)})$.

By the inductive assumption, and the fact that $B_j$ is of radius $R^{1-\delta}$, we have
$$\int_{B_j}Br_{2\alpha}Ef_{I,j,trans}^{3.25}\lesssim C_\varepsilon R^{(1-\delta)\varepsilon}R^{(1-\delta)\delta_{trans}\log(4\alpha\mu K^\varepsilon)}(\displaystyle\sum_\tau\int|f_{\tau,j,trans})^{(3/2)+\varepsilon}$$
Then we need to sum over $j,I$. $|I|\sim2^{K^2}$ is a constant depending on $\varepsilon$, while summing over $j$ is to count the number of $B_j$ a tube can transversely intersect in $W$, which is $R^{ O(\delta_{deg})}$ by Lemma \ref{transct}. Therefore, we obtain
$$\displaystyle\sum_{j,I}(\displaystyle\sum_\tau\int|f_{\tau,j,trans}|^2)^{(3/2)+\varepsilon}\lesssim R^{ O(\delta_{deg})}(\displaystyle\sum_\tau\int|f_\tau|^2)^{(3/2)+\varepsilon}.$$
Thus, we obtain
$$\int_{B_R}Br_\alpha Ef^{3.25}\leq R^{ O(\delta_{deg}-\delta\varepsilon+\delta_{trans})}C_\varepsilon R^\varepsilon R^{\delta_{trans}\log(\alpha\mu K^\varepsilon)}(\displaystyle\sum_\tau\int|f_\tau|^2)^{(3/2)+\varepsilon}.$$
Choosing $\delta_{deg},\delta_{trans}$ carefully, we then obtain (\ref{3.1.2}).
\end{proof}

\begin{proof}[Proof of Lemma \ref{zerodom}]
The argument is similar in spirit to that of Lemma \ref{cell}.

Suppose $x\in B_j\cap W$. We may assume that $x$ is $\alpha$-broad for $Ef$ and that $|Ef(x)|\geq R^{-900}\displaystyle\sum_\tau\|f_\tau\|_2$. Furthermore, we may consider only one subset of caps
$$I\coloneqq\{K^{-1}\text{-caps }\tau:|Ef_{\tau,j,tang}(x)|\leq K^{-100}|Ef(x)|\}.$$
Then, $I^c$ contains caps $\tau$ such that $|Ef_{\tau,j,tang}(x)|\geq K^{-100}|Ef(x)|$. If $I^c$ contains two non-adjacent caps $\tau_1,\tau_2$, then
\begin{equation}\label{zerodombil}
    |Ef(x)|\leq K^{50}|Ef_{\tau_1,j,tang}(x)|^{1/2}|Ef_{\tau_2,j,tang}(x)|^{1/2}\leq K^{100}Bil(Ef_{j,tang})(x).
\end{equation}
Then (\ref{zerodomineq} trivially holds. Thus, we may assume that any two caps in $I^c$ are adjacent, implying that $|I^c|\leq O(\mu)$ because the caps are at most $K^{-1}$ separated and the radius of each cap is at most $\mu^{1/2}K^{-1}$. Choosing $\alpha,\mu$ carefully, we have
$$\displaystyle\sum_{\tau\in I^c}|Ef_\tau(x)|\leq O(\mu\alpha)|Ef(x)|\leq(1/10)|Ef(x)|,$$
where $f_I=\displaystyle\sum_{\tau\in I}f_\tau$. Therefore, $|Ef_I(x)|\geq(9/10)|Ef(x)|$.

Next, let's decompose $Ef_I$ into tangential and transverse contributions. If $T\in\T$ and $T$ intersects $B_j\cap W$, then $T\in\T_{j,trans}\cup\T_{j,tang}$. On the other hand, if $T\cup B_j\cup W=\emptyset$, then $|f_{\tau,T}(x)= O(R^{-1000}\|f_\tau\|_2)$ by wave packet decomposition. Hence, for any cap $\tau$, we have
$$|Ef_\tau(x)|\leq|Ef_{\tau,j,trans}(x)|+|Ef_{\tau,j,tang}(x)+ O(R^{-990}\|f_\tau\|_2).$$
Summing over $\tau\in I$, we obtain
\begin{equation}\label{zerodom1}
    |Ef_I(x)|\leq|Ef_{I,j,trans}(x)|+(\displaystyle\sum_{\tau\in I}|Ef_{\tau,j,tang}(x)|)+ O(R^{-990}\|f_\tau\|_2).
\end{equation}

By definition of $I$, we have $\displaystyle\sum_{\tau\in I}|Ef_{I,j,tang}(x)|\leq K^{-98}|Ef(x)|$. Furthermore, we have $|Ef_I(x)|\geq(9/10)|Ef(x)|$, so
$$(9/10)|Ef(x)|\leq|Ef_I(x)|\leq|Ef_{I,j,trans}(x)|+K^{-98}|Ef(x)|+ O(R^{980}\displaystyle\sum_\tau\|f_\tau\|_2).$$

Moreover, by the assumed loose bound $|Ef(x)|\geq R^{-900}\displaystyle\sum_\tau\|f_\tau\|_2$ from start, we obtain $|Ef(x)|\leq(3/2)|Ef_{I,j,tang}(x)|$. Now it remains to prove that $x$ is also $2\alpha$-broad for $Ef_{I,j,tang}(x)$. For each $\tau\in I$, it suffices to prove that
\begin{equation}\label{trans2}
    |Ef_{\tau,j,trans}(x)|\leq1.1\alpha|Ef(x)|.
\end{equation}
Then $|Ef_{\tau,j,trans}(x)|\leq(33/20)\alpha|Ef_{I,j,trans}(x)|\leq2\alpha|Ef_{I,j,trans}$.

To see (\ref{trans2}), we first observe from (\ref{zerodom1}) that
$$|Ef_{\tau,j,trans}(x)|\leq|Ef_\tau(x)|+|Ef_{j,tang,\tau}(x)|+ O(R^{-990}\|f_\tau\|_2).$$
Since $\tau\in I$, we have $|Ef_{\tau,j,tang}(x)|\leq K^{-100}|Ef(x)|$ by definition of $I$. Moreover, we assume that $x$ is $\alpha$-broad for $Ef$, so
$$|Ef_{\tau,j,trans}(x)|\leq\alpha|Ef(x)|+K^{-100}|Ef(x)|+ O(R^{-990}\|f_\tau\|_2).$$

We use the assumed loose bound $|Ef(x)|\geq R^{-900}\displaystyle\sum_\tau\|f_\tau\|_2$ again and choose $\alpha\geq K^{-\varepsilon}$. Then we obtain the desired bound (\ref{trans2}).
\end{proof}

\begin{remark}
    An alternative broad norm can be applied to simplify the proof to some extent.
\end{remark}

% \section{Alternative broad norms-Kaiyi}
% An alternative broad norm is introduced, thus simplify theorem 3.8 to some extent.

% \begin{definition}
%     For a parameter $A$, the $k$-broad norm is defined as
%     $$\mu_{Ef}\coloneqq\displaystyle\min_{V_1,\dots,V_A}(\displaystyle\max_{\substack{\tau:\text{ for all }a\\\text{Ang}(G(\tau),V_a)>K^{-1}}}\int_{B_{K}}|Ef_\tau|^{3.25}),$$
%     where $V_1,\dots,V_A$ are any $(k-1)$-subspaces, Ang$(\omega_1,\omega_2)$ is the angle between $\omega_1,\omega_2$, and $G(\tau)$ is the direction of tubes corresponding to cap $\tau$.
% \end{definition}

% Then we can define the broad norm by
% % $$Br_AEf(x)\coloneqq\displaystyle\min_{V1,\dots,V_A}(\displaystyle\max_{\substack{\tau:\text{ for all }a\\\text{Ang}(G(\tau),V_a)>K^{-1}}}|Ef_\tau(x)|).$$

% $$\|Ef\|_{BL_A^{3.25}}\coloneqq\mu_{Ef}.$$

% It is worth noting that the broad norm is not a norm, but it does satisfy the following "quasi"-triangle inequality\cite{hong2022broom}
% $$\|E(f+g)\|_{BL_{A_1+A_2}^p}\lesssim\|Ef\|_{BL_{A_1}^p}+\|Ef\|_{BL_{A_2}^p}.$$

\section{Geometry Estimates} \label{sec:geometry}

\subsection{Comments on the transversal estimate.}

As remarked earlier, Lemma \ref{lem: transverse} is conceptually straightforward. A line ``transverse'' to a degree $D$ variety (i.e. it does not lie completely inside it) intersects it in at most $D$ points. To pass to a ``thickened'' statement, one chooses a scale $R$ and a resolution $R^{1-\delta}$. The line is thickened to a $R^{1/2 + \delta}$-tube $T$, the
angle between the line and variety is thickened to $R^{-1/2 + 2\delta}$, the variety is thickened to its $R^{1/2 + \delta}$-neighborhood $W$, and intersection points with the variety are thickened to balls $B_j$ where $T$ intersects $W$ transversally. See Figure \ref{fig:transversal} for a schematic. 

The proof of Lemma \ref{lem: transverse} is technical but not as deep as Lemma \ref{lem: tangent}. Any reasonable upper bound on the number of sets $\mathbb T_{j, trans}$ containing a given tube $T$ is sufficient to close the induction in \cite[Theorem 3.1]{restest}.

The various lemmas leading up to the proof of Lemma \ref{lem: transverse} are similar enough to read like a few unwrapped steps of an induction. In the follow-up paper \cite{Guth_polypart_2}, the higher-dimensional analogue is proved with an induction on dimension (of the variety). When unwrapped in the case of an algebraic surface, one recovers the proof of Lemma \ref{lem: transverse}. See \cite[Lemma 5.7]{Guth_polypart_2} for details. 

\begin{figure}[ht]
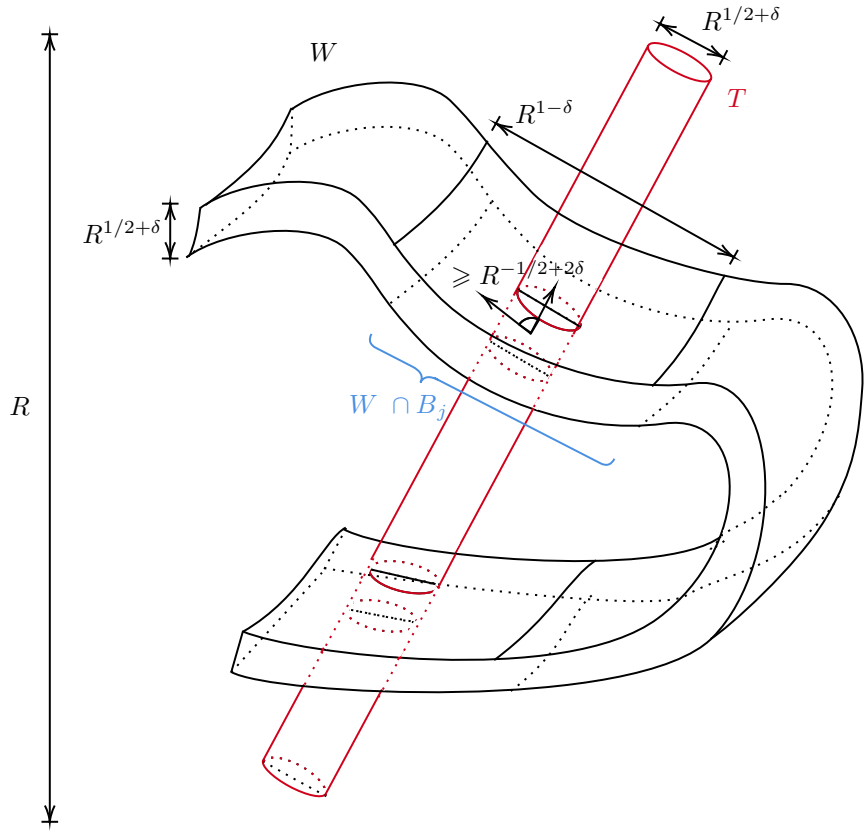
 
    \centering
    \input Pictures/Geometry_section/transversal_tikz.tex
    \caption{Schematic of Lemma \ref{lem: transverse}}
    \label{fig:transversal}
\end{figure}

\subsection{Comments the tangential estimate} 
The tangential estimate Lemma \ref{lem: tangent} plays a more precise role than Lemma \ref{lem: transverse}.
We recall the key harmonic analysis estimate Lemma 
\cite[Proposition 3.9]{restest}
\begin{equation} \label{eq: tang_estimate}
        \int_{B_j} |\mathrm{Bil}(E f_{j, tang})|^{3.25} \lesssim R^{O(\delta)} \left(\int_\tau |f_\tau|^2 \right)^{3/2}. 
\end{equation}
%------ Is this part necessary? -------
As Guth explains, one can run a relatively standard C\'ordoba $L^4$ argument and interpolate between $L^2$ and $L^4$ to get
\begin{equation} \label{eq: tang_estimate_1}
    \int_{B_j} |\mathrm{Bil}(E f_{j, tang})|^p \lesssim  R^{O(\delta)} R^{\frac{5}{2} - \frac{3}{4}p} (\sum_\tau \| f_{\tau, j, trans}\|_2^2)^{p/2},
\end{equation}
where $2 \leq p \leq 4$. Lemma \ref{lem: tangent} says that each $f_{\tau, j, tang}$ is supported in at most $R^{O(\delta)}R^{1/2}$ different caps $\theta$. Applying approximate orthogonality of the $f_T$, we obtain $\|f_{\tau, j, trans}\|_2^2 \lesssim R^{O(\delta)} R^{-1/2}$. Combining this with H\"older bounds the right hand side of \eqref{eq: tang_estimate_1} by the right hand side of \eqref{eq: tang_estimate}.
% If $R^{1/2 + O(\delta)}$ in Lemma \ref{lem: tangent} were replaced with
% $R^A$ (so each $f_{\tau, j, tang}$ is supported in $R^A$ different caps $\theta$) then one may 
% apply the approximate orthogonality of the wave packets $f_T$ and obtain $\|f_{\tau, j, trans} \|_2^2 \lesssim R^{-A}$. So using this, Holder, and assuming $p \geq 3$, 
% \begin{align*}
%     \int_{B_j} |\mathrm{Bil}(E f_{j, tang}|^p 
%     &\lesssim R^{O(\delta)} R^{\frac{5}{2} - \frac{3}{4}p - A(\frac{p}{2} - \frac{3}{2})} (\sum_\tau \|f_{\tau, j, trans} \|_2^2)^{3/2}.
% \end{align*}
One can even check that the example in Section \ref{sec: counterexample} is sharp for \cite[Proposition 3.9]{restest}.

%--------------------
Because Lemma \ref{lem: tangent} is so fundamental to the main result of the paper, we fill in all the details of the proof along with some visual aids. 

\subsection{Proof of the tangential estimate}

Fix a scale $L$ and degree $D$ variety $Z$.
After localizing to a ball $B$ of radius $L$, we are interested in those radius 1 tubes $T$ which intersect $Z$ and stay within an angle $1/L$ of the tangent plane $T_p Z$. Call this collection $\mathbb T_{tang}$. Two tubes point in different directions if the angle between their central axes is at least $1/L$. Lemma \ref{lem: tangent} is a consequence of the following rescaled version. 
\begin{lemma} \label{lem: tang_reduced}
    Any subcollection $\mathbb T \subset \mathbb T_{tang}$ with tubes in pairwise different directions satisfies 
    \begin{equation*}
    |\mathbb T| \lesssim D^2 \log^2(L) L.
    \end{equation*}.
\end{lemma}
To recover Lemma \ref{lem: tangent}, choose $L = R^{1/2 - 2 \delta}$ and $D = R^{\delta_{deg}}$. Scaling the ball $B$ by $R^{1/2 + \delta}$, we see 
that the subcollection $\mathbb T \subset \mathbb T_{j, tang}$ which is $R^{-1/2 + 2 \delta}$ angle separated has $|\mathbb T| \lesssim D^2 \log^2(R^{1/2 - 2 \delta}) R^{1/2 - 2 \delta} = R^{1/2 + O(\delta)}$. Let $\mathbb T'$ be the tubes 
representing each direction in $\mathbb T_{tang}$. These are  $R^{-1/2}$ angle separated since the $R^{-1/2}$-caps of the paraboloid are $R^{-1/2}$ angle separated. 
For each tube  $ T \in \mathbb T$, there can be at most $R^{O(\delta)}$ tubes which are $R^{-1/2}$-separated from each other but not $R^{-1/2 + 2 \delta}$ separated from $T$. Thus 
$|\mathbb T'| \leq R^{O(\delta)} |\mathbb T|$. 

\begin{remark}[Lemma \ref{lem: tang_reduced} is sharp]
    As applied in \cite{restest}, the degree $D$ is small compared to $L$. So the first test of Lemma \ref{lem: tang_reduced} should take $Z$ to be low degree, like the hyperplane $Z = \{x_3 = 0\}$. 
    Take a collection $\mathbb T$ of $1/L$-angle separated 1-tubes whose central axes lie on $Z$. There are $L$ such tubes so Lemma \ref{lem: tang_reduced} is sharp at $D = 1$ up to factors of $\log L$. 
    It turns out that the analogous estimate in higher dimensions is far more difficult, but it has led to the most progress on restriction in higher dimension. We will discuss this later in the section.
\end{remark}

The two main ingredients of the proof are
\begin{enumerate}
    \item The Wolff hairbrush argument. 
    \item An estimate on the volume of the neighborhood of a variety, by Wongkew \cite{volneigh}. 
\end{enumerate}

Item 1 is an argument introduced by Wolff \cite{wolff_improved_kakeya_hairbrush} in the study of the Kakeya maximal function. It will be evident in the proof of Lemma \ref{lem: tang_reduced} why this argument is suggestively named the \emph{hairbrush} argument. Item 2 is: 

\begin{theorem}[Wongkew \cite{volneigh}] \label{thm: wongkew}
    Let $R$ be an $n$-dimensional rectangular grid of unit cubes with dimension $R_1 \times \cdots \times R_n$, where $1 \leq R_1 \leq \cdots \leq R_n$. Suppose $P$ is a non-zero polynomial of degree $D$. Then the number of cubes of the grid that intersect $Z(P)$ is at most $C_n D \prod_{j=2}^n R_j$. 
\end{theorem}
The proof of Theorem \ref{thm: wongkew} is an induction on $n$, exploiting the vanishing of a degree $D$ univariate polynomial to count the intersection of $Z(P)$ with $(n-1)$-faces of the grid. We refer the reader to \cite[Theorem 4.8]{restest} for the details. 

\begin{proof}[Proof of Lemma \ref{lem: tang_reduced}]
    If $T \in \mathbb T_{tang}$, then $T \cap B$ lies in the $O(1)$-neighborhood $N_{O(1)} Z$ of $Z$. This is straightforward, and the case of a ray coming off the $x$-axis at an angle $1/L$ in the plane gives the main idea. 

    It is convenient to write $|\mathbb T| = \beta L$ and show that $\beta \lesssim D^2 \log^2 L$. 
    The idea of the hairbrush argument is to extract a large structured subcollection $H$ of $\mathbb T$ which can be estimated from below. The subcollection $H$ will end up looking like a hairbrush with bristles extending radially from a central axis, more or less making the same angle with the central axis. With all this structure, the volume of $H$ can be estimated quite well from below. In the absence of more information about the tubes, a reasonable upper bound for the volume of $H$ is the volume of the box containing it. 
    However the tubes in $H$ also live in the $O(1)$-neighborhood of $Z$ so we can upper bound the volume of $H$ by upper bounding the piece of the $O(1)$-neighborhood of $Z$ inside the box. This is precisely the estimate provided by Theorem \ref{thm: wongkew}. Playing the upper and lower bounds against each other, one obtains $\beta \lesssim D^2 \log^2(L)$. 

    Now we execute this outline. Cover $N_{O(1)} \cap B$ with unit cubes $Q$. By Theorem \ref{thm: wongkew}, there are $\lesssim D L^2$ such $Q$. 
    We want to count the triples $(Q, T_1, T_2)$, where $T_1$ and $T_2$ are incident to $Q$ and trim the collection to end up with a hairbrush. 
    We have $\sum_Q \sum_T 1_{Q \cap T \neq \emptyset} \sim \sum_T L \sim |\mathbb T| L \sim \beta L^2$, since each of the $\beta L$ tubes $T$ intersect $\sim L$ cubes $Q$. On the other hand, Cauchy-Schwarz gives 
    \begin{align*}
        \left(\sum_Q \sum_T 1_{Q \cap T \neq \emptyset}\right)^{2} &\leq \sum_Q 1 \cdot \sum_{Q} \sum_{T_1, T_2} 1_{Q \cap T_1 \neq \emptyset} 1_{Q \cap T_2 \neq \emptyset}  \\
        &\lesssim DL^2 \cdot \# \{(Q, T_1,T_2)\}.
    \end{align*}
    So $\# \{(Q, T_1,T_2)\} \gtrsim \beta^2 D^{-1} L^2$. This is the correct bound since on average each cube intersects $\beta D^{-1}$ tubes, so one expects $\# \{(Q, T_1,T_2)\} \sim \# Q \cdot (\beta D^{-1})^2 \sim \beta^2 D^{-1} L^2$. 
    The angles $\angle(T_1, T_2)$ between (the central axes of) tubes $T_1$ and $T_2$ range between $1/L$ and $\pi/2$. Group the triples into $\sim \log L$ dyadic blocks according to $\angle(T_1, T_2)$ and pick the block with $\angle(T_1, T_2) \in [\theta, 2\theta]$ that is most popular. So there are $\gtrsim \beta^2 D^{-1} L^2 (\log L)^{-1}$ triples $(Q, T_1, T_2)$ with $\angle(T_1, T_2) \in [\theta, 2\theta]$. 

    There are $\beta L$ tubes $T_1$, so by the pigeonhole principle there is a $T_1$ common to $\gtrsim \beta D^{-1} L (\log L)^{-1}$ tuples $(Q, T_1, T_2)$ with $\angle(T_1, T_2) \in [\theta, 2\theta]$. Then the hairbrush $H$ (with stem $T_1$) is the union of the tubes in these tuples. 
    Now we use the structure to lower-bound $|H|$. A fixed tube $T_2$ can appear in $\lesssim \theta^{-1}$ triples since $\angle(T_1, T_2) \lesssim \theta$.
    So there are $ \gtrsim \beta D^{-1} L (\log L)^{-1} \theta$ ``bristles'' $T_2$ on $H$. 
    
    We can't quite say that $|H| \sim \#(\text{tubes $T_2$ in $H$}) L$ since the tubes may overlap. However the overlap is minor because of the $1/L$-angle separation, so we can show the next best thing,
    \begin{lemma} \label{lem: standard_2dim_Kak_reduction}
        With $H$ as above, $|H| \gtrsim (\log L)^{-1} \#(\text{tubes $T_2$ in $H$}) L$.
    \end{lemma}
    
    The hairbrush $H$ can be fit into a cylinder with length $L$ and radius $\theta L$. Recalling that $H \subset N_{O(1)} Z$, we may apply Theorem \ref{thm: wongkew} to get $|H| \lesssim DL^2 \theta$. Comparing this upper bound for $|H|$ to the lower bound in Lemma \ref{lem: standard_2dim_Kak_reduction}, we get $\beta \lesssim D^2 \log^2 L$.     
\end{proof}

\begin{figure}[ht]
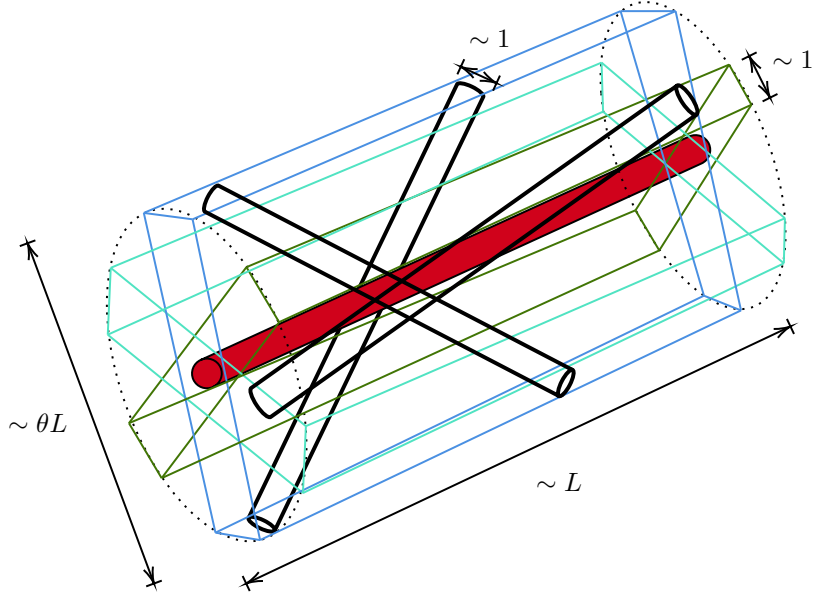
 
    \centering
    \input Pictures/Geometry_section/Wolff_hairbrush_tikz.tex
    \caption{Schematic drawing of the hairbrush $H$. The slabs $S$ are the blue, green, and cyan prisms. The central tube is red and the other tubes in the hairbrush are outlined in black. The other tubes intersect the central tube in an angle $\sim \theta$.}
    \label{fig:Wolff_hairbrush}
\end{figure}

Now we give the proof of Lemma \ref{lem: standard_2dim_Kak_reduction} which is relatively standard. Divide $H$ into $\sim \theta L$ many slabs $S$ of thickness $\sim 1$ as in the Figure \ref{fig:Wolff_hairbrush}. Outside the $\theta/10$-neighborhood $C$ of the stem, each point belongs to $\sim 1$ slabs. Then 
\begin{equation*}
    \#\{T \subset S\} L \sim \int_{S \cap (H \setminus C)} \sum_{T \subset S} \chi_T.
\end{equation*}
This holds since for any tube $T \in H$ one has $|T \cap C| \lesssim \theta L/10$. Hence $|T \cap (H \setminus C)| \gtrsim L(1 - \theta/10) \gtrsim L$. 

Applying Cauchy-Schwarz, we get 
\begin{equation}\label{eq:CS-step_2dim_Kak}
\#\{T \subset S\} L \lesssim |S \cap (H \setminus C)|^{1/2} \|\sum_{T \subset S} \chi_T \|_{L^2(S)}. 
\end{equation}
We are very comfortable with the norm on the RHS, since it is essentially Kakeya in dimension 2. To estimate it, write 
\begin{align*}
    \int_S (\sum_{T \subset S} \chi_T)^2 &= \int_S \sum_{T, T' \subset S} \chi_T \chi_{T'} \\
    &= \sum_{T, T' \subset S} |T \cap T'| \\
    &\lesssim \log L \sum_{\substack{T, T' \subset S \\ \angle(T, T') \in [\phi, 2\phi]}} |T \cap T'|.
\end{align*}
Since $T$ and $T'$ meet at an angle $\sim \phi$, plane geometry gives $|T \cap T'| \lesssim \phi^{-1}$.
For fixed $T \subset S$, there are at most $\phi L$ tubes $T' \subset S$ meeting $T'$ at angle $\phi$ due to the $L^{-1}$ angle separation. Hence 
\begin{align*}
    \sum_{\substack{T, T' \subset S \\ \angle(T, T') \in [\phi, 2\phi]}} |T \cap T'| \lesssim \sum_{T \subset S} (\phi L) \phi^{-1} = \#\{T \subset S\} L.
\end{align*}

Plugging this into \eqref{eq:CS-step_2dim_Kak}, we get $\#\{T \subset S\} L \lesssim (\log L)^{1/2} L^{1/2} |S \cap (H \setminus C)|^{1/2} \#\{T \subset S\}^{1/2}$. Rearranging, $|S \cap (H \setminus C)| \gtrsim (\log L)^{-1} \#\{T \subset S\} L$. Finally we sum over the slabs $S$ and use that outside $C$ each point belongs to $\sim 1$ slabs to get 
\begin{align*}
    |H \setminus C| \gtrsim (\log L)^{-1} \#\{\text{tubes } T \text{ in } H\},
\end{align*}
implying the result.

\subsection{The Polynomial Wolff Axioms and restriction in higher dimensions}

In \cite{Guth_polypart_2}, Guth extends the techniques of \cite{restest} to higher dimensions. As of 2017, this gave the best restriction estimates in dimension $n \geq 4$. We state the restriction conjecture in dimension $n$. 
\begin{conjecture}\label{conj: high_dim_restriction}
    Let $\mathbb P^{n-1} \subset \R^n$ be the truncated paraboloid
    \begin{equation*}
        \mathbb P^{n-1} = \{(\omega', |\omega'|^2) : |\omega'| \leq 1\}.
    \end{equation*}
    Define the extension operator for $\mathbb P^{n-1}$ by 
    \begin{equation*}
        Ef(x) := \int_{\mathbb P^{n-1}} e^{i \omega x} f(\omega) d\mathrm{vol}_{\mathbb P^{n-1}}(\omega).
    \end{equation*}
    Then 
    \begin{equation*}
        \| Ef \|_{L^p(\R^n)} \lesssim \| f\|_{L^p(\mathbb P^{n-1})}
    \end{equation*}
    holds for $p > \frac{2 n }{n-1}$.
\end{conjecture}

In \cite{Guth_polypart_2}, Guth conjectures an analogue of Lemma \ref{lem: tangent} that following similar techniques to \cite{restest}, would translate to even further progress on the restriction conjecture in higher dimensions. 
\begin{conjecture}\label{conj: poly-wolff}
    Let $Z$ be an $m$-dimensional variety in $\R^n$ of degree at most $D$. If $\mathbb T$ is a collection of $1/L$-angle separated tubes where tangency to $Z$ is as in the dimension 3 case Lemma \ref{lem: tangent}, we want 
    \begin{equation*}
        \#\{T \in \mathbb T \text{ tangent to } Z\}  \lesssim_{n, D, \epsilon} L^{m-1 + \epsilon}
    \end{equation*}
    for any $\epsilon > 0$. 
\end{conjecture}
The conclusion of Conjecture \ref{conj: poly-wolff} has come to be known as the Polynomial Wolff Axioms. This name originates from the study of the Kakeya conjecture, which one may regard as the restriction conjecture averaged over wave packets. Often one assumes that the tubes in a Kakeya set satisfy some Wolff Axioms and prove conditional estimates on its dimension. 

%[Explain a bit about usual Wolff axioms...]
Conjecture \ref{conj: poly-wolff} matches Lemma \ref{lem: tang_reduced} in the case $m = 2$. Katz and Rogers \cite{KatzRogersPolyWolff} were able to resolve Conjecture \ref{conj: poly-wolff} completely by using some deep results in geometry and logic. Hickman and Zahl plugged this result into the proof in \cite{Guth_polypart_2} and, as Guth suspected, obtained improved restriction estimates for $n \geq 4$ \cite{hickman2020note}. More general polynomial Wolff axioms were verified for curved tubes that correspond to the phase functions of certain H\"ormander operators in \cite{guo2023dichotomy}, giving further improvements to restriction. To the best of the author's knowledge, this result is still the state of the art. 

\subsection{Improvements to restriction in dimension 3}

Hong Wang \cite{WangBrooms} built on the arguments of Guth \cite{restest} to obtain the restriction estimate as it appears in Theorem \ref{thm: restriction_result} with $p > 3 + \frac{1}{13}$.  The broad estimate, Theorem \ref{broadest}, was improved to the form 
\begin{equation*}
    \|\text{Br}_{K^{-\varepsilon}} Ef\|_{L^{p}(B_R)}\leq C_\varepsilon R^\varepsilon \|f\|_2^{2/p}\|f\|_\infty^{1-2/p},
\end{equation*}
for $p \geq 3 + 3/13$. In \cite{WangBrooms}, the norm $\|f\|_\infty$ is actually replaced by an $L^2$ average over wave packets. This serves a similar role but is better behaved. It is no longer possible to estimate the cellular part of $Ef$ via induction on the radius $R$. Instead, the cellular part of $Ef$ needs to be broken into a local and global part. The local part is handled via induction on the radius, but the analysis of the global part is more delicate and uses a new geometric object called a `broom'. 

Wang and Wu \cite{wang2022improvedrestriction} built on \cite{WangBrooms} to improve the restriction estimate in dimension 3 to $p > 3 + 3/14$, which to the author's knowledge is the state of the art. The main idea was to apply a refined (in the sense of involving two different scales) Wolff hairbrush argument to estimate wave packets of $Ef$ passing through concentrated cells, and to apply the refined decoupling theorem from \cite{Falconer_refineddecoupling} to estimate wave packets passing through cells that are spread out.

%In the years since \cite{restest}, progress on the restriction conjecture has come primarily from improving incidence estimates between tangential tubes and varieties (which has come to be known as the poylnomial Wolff axioms), and more carefully exploiting these available estimates. 
%In \cite{restest}, 

%\section{Progress since this paper}
%Will have talked about Polynomial Wolff axioms in ``The Geometry'' section. Should talk a bit about how they're used later on and why they've become the focus. 

%%%%%%%%%%%%%%%%%%%%%%%%%%%%%%%%%%%%%%%%%%%%%%%
%%%%%%%%%%%%%%%%%%%%%%%%%%%%%%%%%%%%%%%%%%%%%%%
%%%%%%%%%%%%%%%%%%%%%%%%%%%%%%%%%%%%%%%%%%%%%%%

\bibliographystyle{alpha}
\bibliography{main}

\end{document}